\documentclass[12pt,twoside,reqno]{amsart}
\usepackage{amsfonts,amsthm,amsmath,amssymb,dsfont}
\usepackage{amsmath,enumitem}
\usepackage{graphicx}
\usepackage{color}
\usepackage[all]{xy}
\usepackage{subfigure}
\usepackage{tikz}
\usetikzlibrary[arrows,shapes,graphs,decorations.pathmorphing,backgrounds,positioning,fit,petri]
\usepackage{ amscd,   eufrak}
\usepackage{algorithm}
\usepackage{algpseudocode}
 \usepackage{amssymb}
 \usepackage{amsmath}
 \usepackage{booktabs}
 \usepackage{mathrsfs}
 \usepackage{amsfonts}
\usepackage{tabularx}
\usepackage{booktabs}
 \usepackage{verbatim}
\usepackage[colorlinks=true,citecolor=blue]{hyperref}
\usepackage{mathptmx, amsmath, amssymb, amsfonts, amsthm, mathptmx, enumerate, color}
\usepackage[a4paper,
            bindingoffset=0.2in,
            left=0.7in,
            right=0.7in,
            top=1in,
            bottom=1in,
            footskip=.25in]{geometry}
\usepackage{graphicx}
\usepackage{epstopdf}
\newtheorem{theorem}{Theorem}[section]

\newtheorem{proposition}{Proposition}[section]
\theoremstyle{definition}
\newtheorem{definition}{Definition}[section]

\numberwithin{equation}{section}

\begin{document}
\setcounter{page}{1}

\vspace*{1.0cm}
\title[Bayesian fuzzy optimization using fuzzy Gaussian process]
{Bayesian fuzzy optimization using fuzzy Gaussian process}
\author[Sourav Das, Debdas Ghosh, Debjani Chakraborty, Pabitra Mitra]{ Sourav Das$^{1,*}$, Debdas Ghosh$^2$, Debjani Chakraborty$^{1}$, Pabitra Mitra$^3$}
\maketitle
\vspace*{-0.6cm}

\begin{center}
{\footnotesize {\it

$^1$Department of Mathematics, Indian Institute of Technology, Kharagpur, India\\
$^2$Department of Mathematical Sciences, Indian Institute of Technology (BHU) Varanasi, Varanasi, India\\
$^3$Department of Computer Science and Engineering, Indian Institute of Technology Kharagpur, Kharagpur, India
}}\end{center}

\vskip 4mm {\small\noindent {\bf Abstract.}
In many real-life problems, decision-making gets complicated due to dual sources of uncertainty, known as randomness and fuzziness or imprecision, which can be challenging for traditional optimization methods. Most of the existing fuzzy optimization techniques that optimize fuzzy-valued objective functions ignore fuzziness, while the probabilistic optimization techniques ignore randomness. To handle this dual source of uncertainty, Kwakernaak introduced the concept of a fuzzy random variable as ``random variables whose values are not real, but fuzzy numbers". This work aims to derive a theoretical background for the Gaussian fuzzy process and fuzzy acquisition functions, which will be used to develop a novel \emph{Bayesian fuzzy optimization} (BFO) technique that optimizes a fuzzy-values objective function. Based on fuzzy random variables, the Gaussian fuzzy process is developed, which is used as a prior belief about the fuzzy-valued objective function in the BFO. Fuzzy acquisition functions are defined to act as a guide for the search process of BFO with the help of posterior fuzzy mean and fuzzy variance. The proposed method demonstrated effective performance in both fuzzy mean-variance portfolio allocation and Indian temperature data analysis, showing robust predictive accuracy and adaptability. The proposed method can have broader applications in various fields like healthcare, material science, agriculture, etc. \\ 

\noindent {\bf Keywords.}
Fuzzy random variables; Gaussian fuzzy random variables; Gaussian fuzzy process; Bayesian fuzzy optimization; Portfolio allocation. }

\renewcommand{\thefootnote}{}
\footnotetext{ $^*$Corresponding author.
\par
E-mail addresses: souravdas.maths@kgpian.iitkgp.ac.in (Sourav Das), debdas.mat@iitbhu.ac.in (Debdas Ghosh), \\ debjani@maths.iitkgp.ac.in (Debjani Chakraborty),  pabitra@cse.iitkgp.ac.in (Pabitra Mitra).
 }

\section{Introduction}\label{introduction}
In many real-life problems, such as financial modeling, medical diagnostic systems, supply chain management, materials science, etc., there are uncertainties that arise not only from the randomness but also from imprecision or vagueness. Randomness refers to uncertainty that arises from inherent variability or unpredictability in a system. Random variables are generally used to model this kind of uncertainty. Imprecision or vagueness refers to uncertainty that arises from a lack of clear, precise information. This type of uncertainty is often associated with situations where boundaries or definitions are not sharp, and information is ambiguous or incomplete. This type of uncertainty is typically modeled using fuzzy sets or fuzzy numbers. Fuzziness is more of an instrument of a descriptive analysis reflecting the past and its implications, whereas randomness is primarily an instrument of an analytical approach that focuses on the future \cite{shapiro2009fuzzy}. To tackle this situation, \emph{fuzzy random variable} (FRV) was first introduced by Kwakernaak \cite{kwakernaak1978fuzzy, kwakernaak1979fuzzy} as ``random variables whose values are not real, but fuzzy numbers". An expanded version of a similar model was also developed by Kruse and Meyer \cite{kruse2012statistics}. Puri et al. \cite{puri1993fuzzy} defined FRV as random fuzzy sets. The concepts of both the fuzzy mean and variance are also defined, and their properties are investigated in these works. The general concept of a fuzzy probability density is discussed extensively in \cite{viertl2011statistical}. The concept of a Gaussian fuzzy random variable and its properties were developed in \cite{feng2000gaussian}, treating FRV as a random element of the Banach space of all left continuous functions on $(0,1]$ and right continuous on $0$, from $[0,1]$ to the space of continuous functions on a unit sphere.

\par The concept of fuzzy valued functions, where the function values are fuzzy numbers or fuzzy sets rather than precise values, has been introduced by Zadeh \cite{zadeh1965fuzzy}. The idea of fuzzy-valued functions was further developed in the context of fuzzy arithmetic and fuzzy optimization by researchers such as Dubois and Prade \cite{dubois1980fuzzy}, who significantly extended fuzzy set theory into various mathematical and applied fields. The earliest work discussing the use of fuzzy sets in decision-making was by Bellman and Zadeh \cite{bellman1970decision}, which directly led to the development of fuzzy optimization. The earliest works were done by Zimmermann \cite{zimmermann1978fuzzy, zimmermann1993applications, zimmermann1975description}, where fuzzy set theory was applied to the linear programming problems and linear multi-objective programming problems where both the objective function and constraints are fuzzy. The collection of works edited in \cite{delgado1994fuzzy, slowinski2012fuzzy} reported many fuzzy optimization methods. The book \cite{lodwick2010fuzzy} provided a comprehensive presentation of some important classes of fuzzy optimization and fuzzy mathematical programming problems. Search-based metaheuristic methods have also been developed for fuzzy optimization. For example, Buckley and Hayashi \cite{buckley1994fuzzy} introduced a fuzzy genetic algorithm to approximately solve fuzzy optimization problems. A fuzzy particle swarm optimization algorithm was proposed by Tian and Li \cite{5158990}, where inertia weights and learning coefficients are adjusted during the search process based on information given by a fuzzy logic controller.

\par Most of the existing work discussed in the previous paragraph concerning fuzzy optimization only considers the uncertainty due to imprecision or vagueness but ignores uncertainty due to randomness. An early work on Bayesian fuzzy inference was done by \cite{635347}, but this field was under-discovered until Viertl and Sunanta  \cite{viertl2013fuzzy} did a work with the same name, where the authors introduced the fuzzy form of a-priori densities, where the prior information is fuzzy. Few works were done in the current days, like \cite{6955803, 8017463}, where they worked on Bayesian fuzzy clustering and fuzzy Bayesian learning, respectively. However, none of them actually highlight how a fuzzy valued function has to be optimized using the Bayesian approach. To address this gap, this work proposes a fuzzy optimization technique based on a Bayesian approach called Bayesian optimization. Bayesian optimization first originated from the work of Kushner \cite{ref:3}, where a Brownian motion stochastic process is assumed as prior for the objective function and then introduces the probability of improvement acquisition function for finding the location of the maximum point of an arbitrary multi-peak curve in the presence of noise. Another acquisition function, Expected improvement, is developed in \cite{ref:9}. Though these are a few early works, Bayesian optimization got more attention after the work of Jones et al. \cite{ref:6}, where the author proposed an \emph{efficient global optimization} (EGO) algorithm for expensive black-box functions. Bayesian optimization has applications in different fields, such as hyper-parameter tuning in machine learning models, robotics, sensor networks, etc. This optimization method has huge potential to be applied successfully in various fields of science and engineering.

\par This work introduces a sequential search-based fuzzy optimization technique for fuzzy-valued objective functions using Bayesian optimization. Based on the concept of a Gaussian fuzzy random variable defined in \cite{feng2000gaussian} and a fuzzy stochastic process defined in \cite{wang1992theory}, the theory of a Gaussian fuzzy process is discussed, and an expression for the posterior fuzzy mean and fuzzy variance is derived. Next, assuming that a fuzzy-valued function that is to be optimized follows a Gaussian fuzzy process as prior. Based on a few observations, the posterior is updated using the prior, and the posterior fuzzy mean and fuzzy variance are used to define an acquisition function called fuzzy upper confidence bound, which balances between exploration and exploitation and guides the search techniques to get the optimal solutions. The contributions of this work are as follows:
\begin{itemize}
    \item Introduction to \emph{Gaussian fuzzy process} (GFP) and derivations of expression for posterior fuzzy mean and fuzzy variance.
    \item Defining fuzzy acquisition functions called fuzzy expected improvement and fuzzy upper confidence bound.
    \item Generalizing the Bayesian optimization technique to a fuzzy scenario named BFO, where it can optimize a fuzzy-valued function.
    \item Application of the proposed technique in two different problems using financial and climate datasets.
\end{itemize}
The advantage of BFO is that it integrates fuzzy logic with Bayesian optimization, making it more suitable to handle fuzzy-valued objective functions effectively. Incorporating GFP as a prior belief about the fuzzy valued objective function, both randomness and fuzziness in the optimization process are taken into consideration as GFP is based on fuzzy random variables. Fuzzy acquisition functions like \emph{fuzzy expected improvement} (FEI) and \emph{fuzzy upper confidence bound} (FUCB) allow BFO to balance exploration and exploitation more effectively. This leads to a more thorough search of the solution space, potentially discovering better optima. This optimization method can also be used when the explicit form of the objective function is not known, but only its fuzzy values at a few points in the search space are available.

\par In Section \ref{preliminaries}, introduction to fuzzy sets and fuzzy numbers, together with fuzzy matrix, fuzzy valued functions are discussed. Also, traditional Gaussian processes and Bayesian optimization, along with acquisition functions, are discussed. Then Section \ref{Fuzzy random variables} introduces fuzzy random variables and expressions for the fuzzy mean, variance, covariance, etc. Also, some properties of them are included and the definition of the fuzzy probability density function is given. Next, Section \ref{Gaussian fuzzy random variables} discusses the definition of a Gaussian fuzzy random variable and a Gaussian fuzzy random vector with the expression for the mean and variance of conditional distributions. Section \ref{Gaussian fuzzy process} discusses the concepts of the Gaussian fuzzy stochastic process and the expression for the posterior fuzzy mean and fuzzy variance. Bayesian fuzzy optimization is developed in Section \ref{Bayesian fuzzy optimization}, which is used to optimize a fuzzy-valued objective function with very few function evaluations. In Section \ref{Application of Bayesian fuzzy optimization}, two applications of the proposed Bayesian fuzzy optimization method are provided. The first application is to solve the fuzzy mean-variance portfolio allocation problem, and the second application is to optimize the unknown fuzzy-valued function on Indian annual temperature data. Problem setups and results are also discussed in this section. The work concludes with Section \ref{conclusion} providing concluding remarks, limitations, and possible future works.

\section{Preliminaries}\label{preliminaries}

In many real-world problems, the data and information available are not always precise or clear. Instead, they often involve a degree of uncertainty or imprecision. To address this, fuzzy logic provides a framework for dealing with such vague and ambiguous information. This section introduces the concepts of fuzzy sets and fuzzy numbers, which extend classical notions of sets and numbers to handle uncertainty and imprecision more effectively. Also, the related concepts that are needed in the following sections are introduced. 
\subsection{Fuzzy sets and numbers}\label{Fuzzy Sets and Numbers}
Fuzzy Sets are a generalization of classical sets, allowing for degrees of membership rather than a binary membership status. Fuzzy set theory was first introduced by Zadeh \cite{zadeh1965fuzzy} as a generalization of a classical set or crisp set, where the author tried to create a framework that is used to deal with imprecise data and modeling of uncertainties.\par
\begin{definition}
A \textit{fuzzy set} $\tilde{A}$ is an ordered pair $\tilde{A}= \{(x,\mu_{\tilde{A}}(x))\hspace{1mm}|\hspace{1mm} x\in \mathcal{X}\}$
, where $\mathcal{X}$  is the universal set and $\mu_{\tilde{A}}: \mathcal{X} \rightarrow [0,1]$ is the membership function.
\end{definition}
\begin{definition}
A subset $A_{\alpha}=\{x \in \mathcal{X} \hspace{1mm}| \hspace{1mm} \mu_{\tilde{A}}(x) \geq \alpha 
 \} , \forall \alpha \in [0,1]$ of $\mathcal{X}$ is said to be \textit{$\alpha$-cut of the fuzzy set} $\tilde{A}$. i.e, the elements belonging to the fuzzy set $\tilde{A}$ with membership degree greater or equal to $\alpha$. The $\alpha$-cut can also be written as $A_{\alpha} = [\underline{A}_{\alpha}, \overline{A}_{\alpha}]$, where $\underline{A}_{\alpha} = \text{inf}(A_{\alpha})$ and $\overline{A}_{\alpha} = \text{sup}(A_{\alpha})$. The set  $S(\tilde{A})=\{ x \in \mathcal{X}\hspace{1mm}| \hspace{1mm}\mu_{\tilde{A}}(x) > 0\}$ is called \textit{support} of the fuzzy set $\tilde{A}$.
\end{definition}

\begin{definition}
    The indicator function $I_{A_{\alpha}}(x)$ of the $\alpha$-cuts of the fuzzy number $\tilde{A}$ is defined by 
    \begin{equation}\label{eq:indicator}
        I_{A_{\alpha}}(x) = \begin{cases}
                            1 &\text{if } \, \mu_{\tilde{A}}(x) \geq \alpha \\
                            0 \quad &\text{otherwise.}\\
\end{cases}
    \end{equation}
\end{definition}
The membership function of $\tilde{A}$ can be written as $\mu_{\tilde{A}}(x) = \text{sup}\{ \alpha I_{\tilde{A}}(x) \hspace{1mm} | \hspace{1mm} \alpha \in [0,1] \}$. 
\begin{definition}
A Fuzzy set $\tilde{A}$ is said to be \textit{convex} fuzzy set if $\forall x_1,x_2 \in \mathcal{X}$ and $\lambda \in [0,1]$, $\mu_{\tilde{A}}(\lambda x_1+ (1-\lambda)x_2)\geq \min\{\mu_{\tilde{A}}(x_1), \mu_{\tilde{A}}(x_2)\}$. And a fuzzy set $\tilde{A}$ is said to be a \textit{normal} fuzzy set if $\exists x \in \mathcal{X}$, such that $\mu_{\tilde{A}}(x)=1$.
\end{definition}
\begin{definition}
A fuzzy subset $\tilde{A}$ of the universal set $\mathcal{X}$ is called a \textit{fuzzy number}, if \\
(a) $\tilde{A}$ is normal, \\
(b) $\tilde{A}$ is convex, and \\
(c) $\mu_{\tilde{A}}$ is upper  semi-continuous function with bounded support.
\end{definition}

Different types of fuzzy numbers can be defined; popular classes include triangular, trapezoidal, Gaussian, and LR fuzzy numbers. Each type has unique properties and applications, providing a range of tools for effectively handling uncertainty and imprecision in fuzzy logic systems. Understanding these types will help in selecting the most appropriate fuzzy number representation for a given problem or analysis.
\begin{definition}\label{TFN}
A fuzzy number $\tilde{A}= (\underline{a}, a, \overline{a})$ is called a \emph{triangular fuzzy number} (TFN) if its membership function is as follows \\
\begin{equation}\label{eq:tfn}
\mu_{\tilde{A}}(x)=\begin{cases}
\frac{x-\underline{a}}{a-\underline{a}} \quad &\text{if } \, x \in [\underline{a},a] \\
\frac{\overline{a}-x}{\overline{a}-a} \quad &\text{if } \, x \in [a, \overline{a}] \\
0 \quad &\text{otherwise}\\
\end{cases}
\end{equation}
\end{definition}

\begin{definition}
    A fuzzy number $\tilde{A}$ is said to be an LR fuzzy number if 
    \begin{equation}\label{eq:LRfn}
        \mu_{\tilde{A}}(x)=
        \begin{cases}
            L\left(\frac{a-x}{\alpha}\right) \quad &\text{if } \, x\leq a , \alpha>0 \\
            R\left(\frac{x-a}{\beta}\right) \quad &\text{if } \, x\geq a , \beta>0, \\
        \end{cases}
    \end{equation}
    where $a$ is the mean value of $\tilde{A}$ and $\alpha, \beta$ are the left and right spreads respectively, and $L$ and $R$ are left and right shape functions which are symmetric and non-increasing on $[0, \infty).$ The LR fuzzy number is denoted as $\tilde{A}= (a, \alpha, \beta)_{LR}$. One can see that a triangular fuzzy number is a special kind of LR fuzzy number.
\end{definition}

LR fuzzy numbers are a fundamental type of fuzzy number used in various applications of fuzzy logic and optimization. They are characterized by their Left and Right shape functions, which define the shape and extent of the fuzzy number's membership across different values. These fuzzy numbers are particularly useful due to their straightforward representation and computational ease.
\subsection{Operations on LR fuzzy numbers}\label{operations}
In real-world situations where precise values are difficult to determine, fuzzy numbers are often used to represent uncertain or imprecise data. Algebraic operations (addition, subtraction, multiplication, and division) must be defined on these fuzzy numbers in order to interact with them in mathematical models and decision-making processes. For that, consider $\tilde{A_1}= (a_1, \alpha_1, \beta_1 )_{LR}$ and $\tilde{A_2}= (a_2, \alpha_2, \beta_2 )_{LR}$ are two triangular fuzzy numbers, then the operations on them are defined as follows \cite{dubois1980fuzzy}:
\begin{itemize}
    \item Addition: $(a_1, \alpha_1, \beta_1 )_{LR} + (a_2, \alpha_2, \beta_2 )_{LR} =  (a_1+a_2, \alpha_1+\alpha_2, \beta_1+\beta_2 )_{LR}$
    \item Subtraction: $(a_1, \alpha_1, \beta_1 )_{LR} - (a_2, \alpha_2, \beta_2 )_{LR} =  (a_1-a_2, \alpha_1+\beta_2, \beta_1+\alpha_2 )_{LR}$
    \item Scalar multiplication : $\lambda (a, \alpha, \beta) = \begin{cases}
        (\lambda a , \lambda \alpha, \lambda \beta)_{LR}  \quad &\text{if } \lambda>0\\
        (\lambda a , -\lambda \beta, -\lambda \alpha)_{LR}  \quad &\text{if } \lambda<0\\
    \end{cases}$
    \item Multiplication: $(a_1, \alpha_1, \beta_1 )_{LR} \times (a_2, \alpha_2, \beta_2 )_{LR} \simeq (a_1a_2, a_1 \alpha_2 + a_2 \alpha_1 - \alpha_1 \alpha_2 , a_1 \beta_2+ a_2 \beta_1 + \beta_1\beta_2)_{LR}$
    \item Division: $(a_1, \alpha_1, \beta_1 )_{LR} / (a_2, \alpha_2, \beta_2 )_{RL} \simeq (\frac{a_1}{a_2}, \frac{a_1\beta_2 + a_2\alpha_1}{a_2^2}, \frac{a_1\alpha_2+a_2\beta_1}{a_2^2})_{LR}$
    \item Maximum: $\tilde{max}((a_1, \alpha_1, \beta_1 )_{LR}, (a_2, \alpha_2, \beta_2 )_{LR}) \simeq (max(a_1, a_2), min(\alpha_1,\alpha_2), max(\beta_1, \beta_2))_{LR}$
    \item Minimum: $\tilde{min}((a_1, \alpha_1, \beta_1 )_{LR}, (a_2, \alpha_2, \beta_2 )_{LR}) \simeq (min(a_1, a_2), max(\alpha_1,\alpha_2), min(\beta_1, \beta_2))_{LR}$
\end{itemize}
Ordering fuzzy numbers is an essential aspect of fuzzy logic and fuzzy set theory, allowing for the comparison and ranking of fuzzy quantities. There can be different orderings of fuzzy numbers; two of them are discussed in the following definition.
\begin{definition} \cite{guang1991fuzzy} 
    Let $\tilde{A}$ and $\tilde{B}$ be two fuzzy numbers. $\tilde{A} \leq \tilde{B}$ if sup$(A_{\alpha}) \leq $ sup$(B_{\alpha})$ and inf$(A_{\alpha} )\leq $ inf$(B_{\alpha}) \hspace{1mm} \forall \alpha \in [0,1]$. Another similar definition is given and reasonable properties are discussed in \cite{wang2001reasonable}, where it is said that $\tilde{A} \leq \tilde{B}$ if sup$(A_{\alpha}) \leq $ sup$(B_{\alpha})$ and inf$(A_{\alpha} )\geq $ inf$(B_{\alpha}) \hspace{1mm} \forall \alpha \in [0,1]$.
\end{definition}

There are different opinions about the distance between two fuzzy numbers; some define it as a crisp value, and some define it as a fuzzy value \cite{guha2010new, voxman1998some}. Though the distance between two fuzzy numbers should have a fuzzy distance in view of the author, for this work, distance is needed to measure the error only; for that, the distance defined in \cite{chaudhur1996metric} will be considered.
\begin{definition}
    Let $\tilde{A}, \tilde{B}$ are two fuzzy numbers, $[\underline{A}_{\alpha}, \overline{A}_{\alpha}]$ and $[\underline{B}_{\alpha}, \overline{B}_{\alpha}]$ are the $\alpha$-cuts $\forall \alpha \in [0,1]$. If $h_{alpha}(A_{\alpha},B_{\alpha})$ represents the Hausdorff distance between the compact intervals $A_{\alpha}$ and $B_{\alpha}$ defined by
    \begin{equation*}
        h_{\alpha}(A_{\alpha},B_{\alpha}) = \max \left( \sup_{x \in A_\alpha} \inf_{y \in B_\alpha} d(x, y), \sup_{y \in B_\alpha} \inf_{x \in A_\alpha} d(x, y) \right).
    \end{equation*}
    then the distance between two fuzzy numbers is defined as 
    \begin{equation}\label{eq:distance}
        h(\tilde{A}, \tilde{B}) = \int_{0}^{1} h_\alpha(A_\alpha, B_\alpha) \, d\alpha
    \end{equation}
\end{definition}
A fuzzy matrix is essentially a matrix whose entries are fuzzy numbers or fuzzy sets. These matrices can represent various types of systems, including fuzzy linear systems, fuzzy relations, etc. The operations performed on fuzzy matrices, such as addition and multiplication, are adapted to handle the fuzzy nature of the elements. This involves using fuzzy arithmetic and membership functions to operate on the fuzzy numbers within the matrix. Later in his work, to develop the theory, the concept of a fuzzy inverse matrix is needed. For the definition of the fuzzy inverse matrix, we refer to
 \cite{dehghan2009inverse}
\begin{definition}
    A fuzzy matrix $\tilde{M}$ is a matrix whose elements are fuzzy numbers, i.e, $\tilde{M}_{ij}= \tilde{m}_{ij}$, where each $\tilde{m}_{ij}$ is a fuzzy number for all $i,j$. A positive fuzzy matrix $\tilde{A}= (A, M, N)$, where $A, M, N$ are the real-valued matrices containing the means, left spreads and right spreads respectively, then its inverse is given by $\tilde{B}= (A^{-1}, -A^{-1}MA^{-1}, -A^{-1}NA^{-1})$. 
\end{definition}

Note that a function $\tilde{f}$ is said to be a fuzzy-valued function if its outputs are not precise, but rather represented by fuzzy numbers. It extends the concept of classical functions to fuzzy situations.

\subsection{Gaussian process}\label{Gaussian Process}
A Gaussian process, $\mathcal{G}\mathcal{P}$, is a generalization of the multivariate Gaussian distributions to infinitely many variables. It is distributed over functions. Formally, we can say that,
\begin{definition}
A Gaussian Process is a collection of random variables, any finite number of which are multivariate Gaussian.
\end{definition}
Assume a function $f : \mathcal{X} \rightarrow{} \mathbb{R}$ follows Gaussian process, i.e., $f(x) \sim\mathcal{G}\mathcal{P}(\mu , \kappa)$, where $\mu(x) : \mathcal{X} \rightarrow{} \mathbb{R} $, defined as $\mu(x) = \mathbb{E}[f(x)]$,  is the mean function and $\kappa : \mathcal{X} \times \mathcal{X} \rightarrow{}  \mathbb{R}$, is covariance or kernel function. Here $\mu(x)$ is the average function value of all functions present in the distribution at the point $x$ and $\kappa(x, x^{'})$ represents the dependence between the function values at different input points. A kernel is usually chosen based on the assumption that two points are more correlated as their distance decreases. \par Consider a finite collection of $n$ points, $x_{1},x_{2},\ldots, x_{n} \in\mathcal{X}$. Then their function values are $f_{i} = f(x_{i})$, for $i = 1,2,\ldots, n$. From the assumption $f_{1:n} = [f_{1}, f_{2}, \ldots , f_{n}]$ are Jointly Gaussian with the mean vector  $\mu_{1:n} = [\mu_{1}, \mu_{2}, \ldots , \mu_{n}]$ and covariance matrix $\mathbf{K}_{i,j} = \kappa(x_{i}, x_{j})$. Now for a new point $x$, its function value $f(x)$ also follows Gaussian distribution with the following mean and variance function.\\
\begin{equation}\label{eq:3.1}
  \mu_{*}(x) = \mu(x) + \mathbf{k}(x, x_{1:n})^{T}\mathbf{K}^{-1}(f_{1:n} - \mu_{1:n})
\end{equation}
\begin{equation}\label{eq:3.2}
  \sigma^{2}_{*}(x) = \kappa(x,x) - \mathbf{k}(x, x_{1:n})^{T}\mathbf{K}^{-1}\mathbf{k}(x, x_{1:n})
\end{equation}
Here, $\mathbf{k}(x, x_{1:n}) = [\kappa(x, x_{1}), \kappa(x, x_{2}), \ldots , \kappa(x, x_{n})]$. This mean and variance are called posterior mean and variance and the distribution is called posterior distribution. The posterior mean $\mu_{*}(x)$ is a kernel-dependent weighted average between the prior $\mu(x)$ and an estimate $f_{1:n}$ derived from the data and posterior covariance $\sigma^{2}_{*}(x)$ is nothing but a term subtracted from prior covariance $\kappa(x,x)$ corresponding to the variance removed by the previously seen data.
\par The kernel function must be positive definite, in the sense that for any finite collection of points, the kernel matrix formed by pairwise evaluation is positive definite. There are several kernel functions, but this work will use the Squared exponential kernel, also known as the Radial basis kernel, and defined as 
\begin{equation}\label{eq:3.3}
  \kappa(x,x^{'}) = s^{2}\exp\left({- \frac{(x - x^{'})^2}{2l^{2}}}\right), 
\end{equation}
where $s$ is the scale factor and $l$ is the length scale.

\subsection{Bayesian optimization}\label{Bayesian Optimization}
Bayesian Optimization is a sequential model-based\footnote{Sequential model-based optimization cycles through the process of fitting models and using them to decide which options to examine.} method for carrying out global optimization of unknown, expensive-to-evaluate, black-box objectives.   
A probabilistic model, which captures our belief about the behavior of the unknown objective function, and an acquisition function, which determines where to sample next, are the two main components of Bayesian optimization. After initializing a prior belief about the objective function $f$, which is usually a Gaussian process, $\mathcal{G}\mathcal{P}(\mu(x), \kappa(x, x^{'}))$ and collecting $n$ sample points $x_{1}, x_{2}, \ldots, x_{n}$ with the function values $f(x_{1}), f(x_{2}), \ldots , f(x_{n})$, posterior distribution is updated using Equation \ref{eq:3.1} and \ref{eq:3.2}, which is used to find the maximizer $x_{*}$ of the acquisition function $\alpha : \mathcal{X} \rightarrow \mathbb{R}$. Then posterior is updated with the modified observed values. See Algorithm~\ref{alg:1}.\\

\begin{algorithm}
\caption{Bayesian Optimization}
\label{alg:1}
\begin{algorithmic}[1]
\State{Assume that the Objective function $f$ follows a prior distribution.}
\State{Observe value of $f$ at $n$ points.}
\While{Condition is true }
\State{Update Posterior distribution on $f$.}
\State{From the current Posterior distribution, find the maximizer of the acquisition function.}
\State{Find the function value of $f$ at the maximizer.}
\EndWhile
\State \textbf{Return} Point giving largest objective function value.
\end{algorithmic}
\end{algorithm}
Until now, only the statistical model has been discussed, which is mainly a Gaussian Process and represents belief about the unknown objective function. However, the procedure to generate the sequence of points in each iteration is not described. Random selection of query points could be possible, but that will be a waste; selection strategies, also known as the acquisition function, that use the posterior model to guide the selection search are used.

\subsection{Acquisition functions}

In Bayesian optimization, the acquisition function is the function that determines how the parameter space should be searched with the help of the posterior distribution. 
Improvement-based acquisition functions favor the points that are likely to produce improvement upon the previously observed best objective function value. Let $f_{m}^{*}=max(f(x_{n}))$ for all $n\leq m$ be the best value after $m$-th iteration. So at the $(m+1) $-th iteration if the query point is $x_{m+1}$ and the objective value is $f(x_{m+1})$ then there will be an improvement upon $f_{m}^{*}$, if $f(x_{m+1})-f_{m}^{*}> 0$ . 
\subsubsection{Probability of Improvement (PI)}
At $(m+1)$-th iteration for an arbitrary point $x\in \mathcal{X}$ the improvement upon  $f_{m}^{*}$ is $f(x)-f_{m}^{*}$. Since $f$ follows the Gaussian process with posterior mean and variance as given in Equation \eqref{eq:3.1} and \eqref{eq:3.2}, given the observations $\mathcal{D}=\{(x_{i}, f(x_{i}))\}_{i=1}^{m}$, probability of improvement will be
\begin{equation}\label{eq:3.4}
\alpha(x) = \mathbb{P}[f(x)-f_{m}^{*}> 0] = 1- \mathbb{P}[f(x)\leq f_{m}^{*}] = 1- \Phi\left(\frac{f_{m}^{*}- \mu_{*}(x)}{\sigma_{*}(x)}\right)
\end{equation}
where $\Phi$ is the Standard normal cumulative distribution function. Recall from Section \ref{Bayesian Optimization} that, the acquisition function is then maximized to find the maximizer as the next query point. So, at $(m+1)$-th iteration the query point will be
\begin{equation}\label{eq:3.5}
x_{m+1}= argmax_{x}\alpha(x) 
\end{equation}
Though this early strategy in the literature, PI \cite{ref:3}, performs well when the target is known, in general PI exploits highly \cite{ref:4} with less exploration, which can lead the search procedure to get stuck in a local optimum. To address this problem, the following acquisition function accounts for the expected improvement. 
\subsubsection{Expected improvement}\label{EI}
The \emph{expected improvement} (EI) acquisition function is defined as the expectation of the improvement of the current functional value over the current best functional value, to be positive
\begin{equation}\label{eq:3.6}
\alpha(x) = \mathbb{E}[f(x)-f_{m}^{*}> 0] =  (f_{m}^{*}-\mu_{*}(x)) \Phi\left(\frac{f_{m}^{*}- \mu_{*}(x)}{\sigma_{*}(x)}\right) + \sigma_{*}(x) \phi\left(\frac{f_{m}^{*}- \mu_{*}(x)}{\sigma_{*}(x)}\right), 
\end{equation}    
where $\Phi$ and $\phi$ are the standard normal cumulative(CDF) and probability density(PDF) function respectively. Similarly to Equation \eqref{eq:3.5}, one can also find the maximizer $x_{m+1}$ in this case. Intuitively, it can be thought of as a weighted sum of the improvement and uncertainty, with weights being the standard normal CDF and PDF. i.e., this acquisition function is balancing between the exploitation near the current best objective value and exploring the points where the uncertainty is high. Uncertainty is high means that, the region, where less data has been observed.\\

\subsubsection{Upper confidence bound }\label{ucb} The \emph{upper confidence bound} (UCB) is a popular optimistic way to balance between exploration and exploitation by considering the weighted sum of posterior mean and variance, defined as
\begin{equation}\label{eq:3.7}
\alpha(x) = \mu_{*}(x) + \beta\sigma_{*}(x), 
\end{equation}
where $\beta$ is an unknown parameter that represents how much preference is given to the exploration while searching for the next query point.\par
 
Bayesian optimization has been applied to various optimization problems successfully among which hyper-parameter tuning is the most popular one. In this work, authors will try to use Bayesian optimization techniques for the Portfolio allocation optimization problem.

\section{Fuzzy random variables}\label{Fuzzy random variables}
Classical random variables incorporate the uncertainty due to randomness in the model but fail to take into consideration the uncertainty associated with imprecision or vagueness. However, in many real-world applications, the model contains uncertainty associated with both sources. To overcome this situation, the concept of a fuzzy random variable has been developed in the literature.
The term ``fuzzy random variable" was first used by Kwakernaak \cite{kwakernaak1978fuzzy, kwakernaak1979fuzzy}, and the basic idea behind the definition is to consider random variables taking a fuzzy number instead of real values. On the other hand, Puri and Ralescu \cite{puri1993fuzzy} defined FRV as a fuzzification of a random set and defined the fuzzy numbers on the more generalized Separable Banach space. Throughout this work, the fuzzy random variable defined by Kwakernak will be considered.\par
Consider $(\Omega, \mathcal{A}, \mathcal{P})$ to be a probability space, where $\Omega$ is the sample space, $\mathcal{A}$ is the event space, and $\mathcal{P}$ is the probability measure. Assume $\mathcal{F}(\mathbb{R})$ to be the collection of all fuzzy numbers in $\mathbb{R}$, i.e., $\mathcal{F}(\mathbb{R})= \{\tilde{A}:\mathbb{R} \rightarrow [0,1] \hspace{1mm}| \hspace{1mm}A_{\alpha} \text{ is compact } \forall \alpha \in [0,1]\}$, where $A_{\alpha}= \{x \in \mathbb{R} \hspace{1mm} | \hspace{1mm} \mu_{\tilde{A}}(x) \geq \alpha 
 , \forall \alpha \in (0,1]\}$ and $A_{\alpha}=closure(S(\tilde{A}))$ for $\alpha=0$. \par
 \begin{definition}
  A \textit{fuzzy random variable} is defined as a mapping $X: \Omega \rightarrow \mathcal{F}(\mathbb{R})$ for all $\alpha \in [0,1] \text{ and } \omega \in \Omega$ such that, the real-valued mappings 
 $\underline{X}_{\alpha}:\Omega \rightarrow \mathbb{R} \text{ defined by } \underline{X}_{\alpha}(\omega)= inf(X(\omega)_{\alpha})  $ and $\overline{X}_{\alpha}:\Omega \rightarrow \mathbb{R}$ defined by $\overline{X}_{\alpha}(\omega)= sup(X(\omega)_{\alpha})  $ are real-valued random variables.
 \end{definition}
\subsection{Fuzzy expectation and variance of fuzzy random variables}
If $X$ is a fuzzy random variable, then its fuzzy expectation $E[X]$ is defined in Equation (\ref{eq:expectation}) and its variance $Var(X)$ is given in Equation (\ref{eq:variance}) in terms of their $\alpha$-cuts, as discussed in \cite{berkachy2021fuzzy, kwakernaak1978fuzzy}:  
\begin{align}\label{eq:expectation}
    (E[X])_{\alpha} = \big[  E[\underline{X}_{\alpha}] , E[\overline{X}_{\alpha}]\big] 
\end{align}
\noindent
The membership function of the variance is given by \cite{berkachy2021fuzzy}
\begin{align}
    \mu_{Var(X)}(x) = \text{sup} \{ \alpha I_{\big[(\underline{Var(X)})_{\alpha} \text{,} (\overline{Var(X)})_{\alpha} \big] }(x) | \alpha \in [0 , 1]\}. 
\end{align}
\noindent
Also, the variance can be written using the alpha cuts 
\begin{align}\label{eq:variance}
(Var(X))_{\alpha} = \big[  (\underline{Var(X)})_{\alpha} , (\overline{Var(X)})_{\alpha}\big],  
\end{align} 
where $(\underline{Var(X)})_{\alpha} = \text{inf}_{X^{'} \in [\underline{X}_{\alpha} , \overline{X}_{\alpha}]} (Var(X^{'}))$ and $(\overline{Var(X)})_{\alpha} = \text{sup}_{X^{'} \in [\underline{X}_{\alpha} , \overline{X}_{\alpha}]} (Var(X^{'}))$. 
Similarly covariance between any two fuzzy random variables $X \text{ and } Y$ is defined as
\begin{align}\label{eq:covariance}
(Cov(X, Y))_{\alpha} = \big[  (\underline{Cov(X , Y)})_{\alpha} , (\overline{Cov(X , Y)})_{\alpha}\big],  
\end{align}
\noindent
where  $(\underline{Cov(X, Y)})_{\alpha} = \text{inf}_{X^{'} \in [\underline{X}_{\alpha} , \overline{X}_{\alpha}] , Y^{'} \in [\underline{Y}_{\alpha} , \overline{Y}_{\alpha}]} (Cov(X^{'} , Y^{'}))$ and \\
$(\overline{Cov(X , Y)})_{\alpha} = \text{sup}_{X^{'} \in [\underline{X}_{\alpha} , \overline{X}_{\alpha}] , Y^{'} \in [\underline{Y}_{\alpha} , \overline{Y}_{\alpha}]} (Cov(X^{'} , Y^{'}))$.

\begin{proposition}
        For any two fuzzy random variables $X \text{ and } Y$, $Z=X+Y$ is a fuzzy random variable with mean $E[Z]= E[X]+E[Y]$ and variance $Var[Z]= Var[X]+Var[Y]+2Cov(X,Y)$. 
\end{proposition}

\subsection{Fuzzy probability density function}
The general definition of a fuzzy probability distribution stated by \cite{viertl2011statistical} is given as  
\begin{definition}\label{fuzzypdf}
    Consider $(\Omega, \mathcal{A}, \mathcal{P})$ to be a probability space. A fuzzy-valued function $\tilde{f}: \Omega \rightarrow \mathcal{F}(\mathbb{R})$, represented by its $\alpha$-cuts $\big[ \underline{f}_{\alpha}(x), \overline{f}_{\alpha}(x) \big]$  is said to be fuzzy probability density function if it satisfies the following conditions:
    \begin{itemize}
        \item $\int_{\Omega}\underline{f}_{\alpha}(x)d\mathcal{P}(x)$ and $\int_{\Omega}\overline{f}_{\alpha}(x)d\mathcal{P}(x)$ exist. 
        \item $1 \in \Big[\int_{\Omega}\underline{f}_{1}(x)d\mathcal{P}(x) , \int_{\Omega}\overline{f}_{1}(x)d\mathcal{P}(x)\Big]$ and the support of the fuzzy integral is in $\mathbb{R}^{+}$
        \item There exists a crisp probability density function $f: \Omega \rightarrow \mathbb{R}^{+}$ with $\underline{f}_{1}(x) \leq f(x) \leq \overline{f}_{1}(x) \forall x \in \Omega$.  
    \end{itemize}
\end{definition}
\begin{figure}
    \centering
        \includegraphics[width=0.5\textwidth]{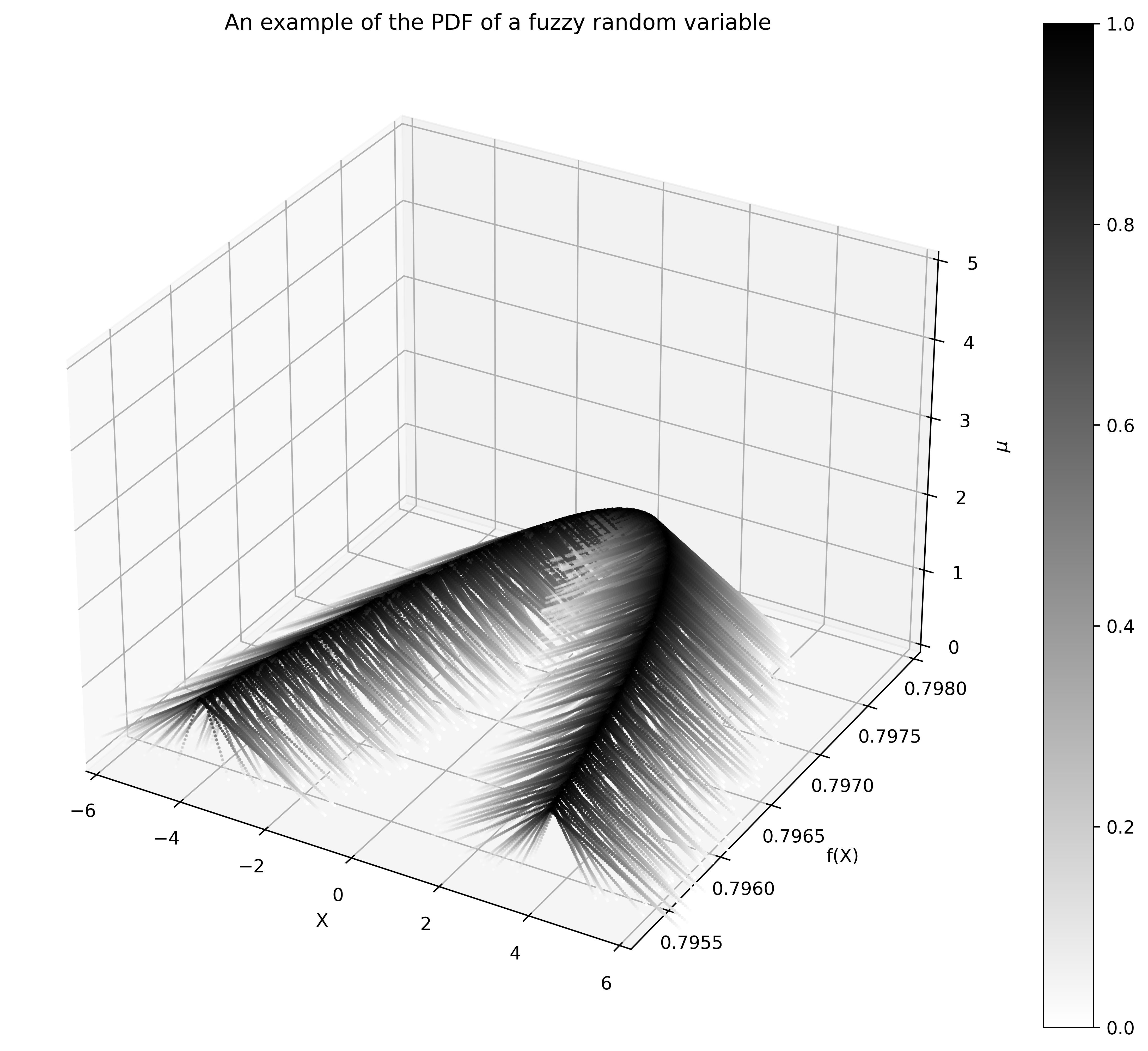}
         \caption{Example of a fuzzy random variable}
         \label{fig:FRV}
\end{figure}
\noindent
Having established the foundational concepts and significance of fuzzy random variables in capturing both randomness and imprecision in complex systems, the next movement is towards a more specialized area within this framework: Gaussian fuzzy random variables. Gaussian fuzzy random variables offer a specific approach that combines the properties of Gaussian distributions with fuzzy logic.

\section{Gaussian fuzzy random variables}\label{Gaussian fuzzy random variables}

The concept and theoretical development of Gaussian fuzzy random variable is derived in the work \cite{feng2000gaussian}. The author provided the expression for the fuzzy mean and variances for the Gaussian fuzzy random variable as well as the Gaussian fuzzy random vector. They have left the space open for developing the Gaussian fuzzy process. Simplified definitions can be stated as 
\begin{definition}
  A fuzzy random variable $X:\Omega \rightarrow \mathcal{F}(\mathbb{R})$ is said to be a Gaussian fuzzy random variable if the corresponding real-valued random variables $\underline{X}_{\alpha} \text{ and } \overline{X}_{\alpha}$ for all $\alpha \in [0,1]$ are Gaussian random variables with known mean and variance.  
\end{definition}
\subsection{Gaussian fuzzy random vector}
$(X^1, X^2, \ldots , X^k) : \Omega \rightarrow (\mathcal{F}(\mathbb{R}))^k $ is said to be a fuzzy random vector if $X^j$'s are fuzzy random variables $\forall j \in \{1, 2, \ldots, k\}$.
\begin{definition}\label{fuzzy_random_vector}
A fuzzy random vector $\mathbf{X} = (X^1, X^2, \ldots , X^k)$ is said to be a $k$-dimensional Gaussian fuzzy random variable if 
\[
Z = \sum_{i=1}^{k} a_i X^i
\]
is a Gaussian fuzzy random variable for every $a_1, a_2, \ldots, a_k \in \mathbb{R}$.

The fuzzy mean vector and fuzzy covariance matrix of $\mathbf{X}$ are given by
\[
m = \big(E[X^1], E[X^2], \ldots, E[X^k]\big), 
\Sigma = (\tilde{\sigma}_{ij}),
\]
respectively, where $\tilde{\sigma}_{ij}$ denotes the fuzzy covariance of $X^i$ and $X^j$ (see Equations~(\ref{eq:expectation}) and (\ref{eq:covariance}) for the $\alpha$-cuts of the mean and covariance).
\end{definition}
\begin{definition}
    The fuzzy probability density function for the fuzzy random vector $X$ is denoted as $F_X$ and defined by its $\alpha$-cuts as $\big[(\underline{F_X})_{\alpha}(x) , (\overline{F_X})_{\alpha}(x)\big]$ $\forall \alpha \in[0,1]$, where $(\underline{F_X})_{\alpha}(x) $ and $ (\overline{F_X})_{\alpha}(x)$ are the joint probability distributions of the real-valued random vectors $\underline{{X}}_{\alpha} = ((\underline{X^{1}})_{\alpha}, (\underline{X^{2}})_{\alpha}, \ldots, (\underline{X^{k}})_{\alpha})$ and $\overline{X}_{\alpha} = ((\overline{X^{1}})_{\alpha}, (\overline{X^{2}})_{\alpha}, \ldots, (\overline{X^{k}})_{\alpha})$. Expressions for $(\underline{F_X})_{\alpha}(x) $ and $ (\overline{F_X})_{\alpha}(x)$ are 
\end{definition}
\begin{align}\label{lower_pdf_frvector}
    (\underline{F_X})_{\alpha}(x) = 
 \frac{1}{(2\pi)^{k/2} |(\underline{\boldsymbol{\Sigma}})_{\alpha}|^{1/2}} \exp\left(-\frac{1}{2} (\underline{X}_{\alpha} - (\underline{m})_{\alpha})^T (\underline{\boldsymbol{\Sigma}})_{\alpha}^{-1} (\underline{X}_{\alpha} - (\underline{m})_{\alpha})\right)
\end{align}
\begin{align}\label{upper_pdf_frvector}
    (\overline{F_X})_{\alpha}(x) = 
 \frac{1}{(2\pi)^{k/2} |(\overline{\boldsymbol{\Sigma}})_{\alpha}|^{1/2}} \exp\left(-\frac{1}{2} (\overline{X}_{\alpha} - (\overline{m})_{\alpha})^T (\overline{\boldsymbol{\Sigma}})_{\alpha}^{-1} (\overline{X}_{\alpha} - (\overline{m})_{\alpha})\right), 
\end{align}
where  $(\underline{m})_{\alpha}, (\underline{\boldsymbol{\Sigma}})_{\alpha}$ and $(\overline{m})_{\alpha}, (\overline{\boldsymbol{\Sigma}})_{\alpha}$ are the mean vectors and covariance matrices for the real-valued Gaussian random vectors $\underline{X}_{\alpha}$ and $\overline{X}_{\alpha}$ respectively.
\subsection{Conditional fuzzy Gaussian random variables}
\begin{theorem}\label{th:conditional_distribution}
    If two sets of fuzzy random variables are jointly fuzzy Gaussian, then the conditional distribution of one set conditioned on the other is again fuzzy Gaussian with conditional fuzzy mean and fuzzy covariance matrix given in terms of their $\alpha$-cuts by $({m_{a|b}})_{\alpha} = [(\underline{m_{a|b}})_{\alpha} , (\overline{m_{a|b}})_{\alpha}]$ and  $({\Sigma_{a|b}})_{\alpha} = [(\underline{\Sigma_{a|b}})_{\alpha} , (\overline{\Sigma_{a|b}})_{\alpha}]$, where $(\underline{m_{a|b}})_{\alpha}$ and $(\underline{\Sigma_{a|b}})_{\alpha}$ are the real valued conditional mean vector and covariance matrix for $(\underline{X_a})_{\alpha}$ conditioned on $(\underline{X_b})_{\alpha}$ and similar for $(\overline{X_a})_{\alpha}$ conditioned on $(\overline{X_b})_{\alpha}$. 
\end{theorem}
\begin{proof}
    Let $\mathbf{X}=(X^1,X^2, \ldots,X^k) $ be a Gaussian fuzzy random vector with fuzzy mean vector $m= (\tilde{m_1} , \tilde{m_2}, \ldots, \tilde{m_k})$ and $\Sigma = \tilde{\sigma}_{ij}$, where $\tilde{m}_i = E[X^i] \hspace{1mm} \forall \hspace{1mm} i \in \{1,2,\ldots, k\}$ (see  Definition \ref{fuzzy_random_vector}).
    \par $\mathbf{X}$ is partitioned into two disjoint subsets $\mathbf{X_a} = (X^1,X^2, \ldots,X^p)$ and $\mathbf{X_b} = (X^{p+1},X^{p+2}, \ldots,X^{k})$, i.e, $X= (X_a , X_b)$. Then the partitioned mean vector is $m = (m_a , m_b)$, where $m_a = (\tilde{m}_1, \tilde{m}_2, \ldots, \tilde{m}_p)$ and $m_b = (\tilde{m}_{p+1}, \tilde{m}_{p+2}, \ldots, \tilde{m}_k)$. Now, the covariance matrix $\Sigma$ is given by 
    $\Sigma = 
    \begin{pmatrix}
        \begin{array}{c c}
            \Sigma_{aa} & \Sigma_{ab} \\
            \Sigma_{ba} & \Sigma_{bb}
        \end{array}
    \end{pmatrix}$
Note that, $\Sigma_{aa}$ and $\Sigma_{bb}$ are symmetric and $\Sigma_{ab} = \Sigma_{ba}^T$. Also, define the precision matrix 
$ \Sigma^{-1} = \Lambda = 
\begin{pmatrix}
    \begin{array}{c c}
        \Lambda_{aa} &  \Lambda_{ab} \\
        \Lambda_{ba} &  \Lambda_{bb}
    \end{array}
\end{pmatrix}
$.
From Equation (\ref{lower_pdf_frvector}) it can be seen that the exponent of $(\underline{F_X})_{\alpha}(x)$ for all $\alpha \in [0,1]$ is given by \\\\

\begin{align*}
-\frac{1}{2} (\underline{X}_{\alpha} - (\underline{m})_{\alpha})^T (\underline{\boldsymbol{\Sigma}})_{\alpha}^{-1} (\underline{X}_{\alpha} - (\underline{m})_{\alpha}) 
&= -\frac{1}{2} ((\underline{X_a})_{\alpha} - (\underline{m_a})_{\alpha})^T (\underline{\boldsymbol{\Lambda}_{aa}})_{\alpha} ((\underline{X_a})_{\alpha} - (\underline{m_a})_{\alpha})\\
&\quad -\frac{1}{2} ((\underline{X_a})_{\alpha} - (\underline{m_a})_{\alpha})^T (\underline{\boldsymbol{\Lambda}_{ab}})_{\alpha} ((\underline{X_b})_{\alpha} - (\underline{m_b})_{\alpha})\\
&\quad -\frac{1}{2} ((\underline{X_b})_{\alpha} - (\underline{m_b})_{\alpha})^T (\underline{\boldsymbol{\Lambda}_{ba}})_{\alpha} ((\underline{X_a})_{\alpha} - (\underline{m_a})_{\alpha})\\
&\quad -\frac{1}{2} ((\underline{X_b})_{\alpha} - (\underline{m_b})_{\alpha})^T (\underline{\boldsymbol{\Lambda}_{bb}})_{\alpha} ((\underline{X_b})_{\alpha} - (\underline{m_b})_{\alpha})
\end{align*}

From the proof given in Section 2.3.1 of \cite{bishop2006pattern}, expression for conditional mean and covariance will be 
\begin{align*}
    (\underline{m_{a|b}})_{\alpha} = (\underline{m_a})_{\alpha} + (\underline{\boldsymbol{\Sigma}_{ab}})_{\alpha}(\underline{\boldsymbol{\Sigma}_{bb}})_{\alpha}^{-1}((\underline{X_b})_{\alpha} - (\underline{m_b})_{\alpha})
\end{align*}
\begin{align*}
    (\underline{\Sigma_{a|b}})_{\alpha} = (\underline{\boldsymbol{\Sigma}_{aa}})_{\alpha} - (\underline{\boldsymbol{\Sigma}_{ab}})_{\alpha}(\underline{\boldsymbol{\Sigma}_{bb}})_{\alpha}^{-1}(\underline{\boldsymbol{\Sigma}_{ba}})_{\alpha}.  
\end{align*}
Similarly, it can be shown that
\begin{align*}
    (\overline{m_{a|b}})_{\alpha} = (\overline{m_a})_{\alpha} + (\overline{\boldsymbol{\Sigma}_{ab}})_{\alpha}(\overline{\boldsymbol{\Sigma}_{bb}})_{\alpha}^{-1}((\overline{X_b})_{\alpha} - (\overline{m_b})_{\alpha})
\end{align*}
\begin{align*}
    (\overline{\Sigma_{a|b}})_{\alpha} = (\overline{\boldsymbol{\Sigma}_{aa}})_{\alpha} - (\overline{\boldsymbol{\Sigma}_{ab}})_{\alpha}(\overline{\boldsymbol{\Sigma}_{bb}})_{\alpha}^{-1}(\overline{\boldsymbol{\Sigma}_{ba}})_{\alpha}.  
\end{align*}
Hence, the conditional fuzzy mean vector and fuzzy covariance matrix can be given by their $\alpha$-cuts as 
\begin{align}\label{fuzzy_conditional_mean}
    ({m_{a|b}})_{\alpha} = [(\underline{m_{a|b}})_{\alpha} , (\overline{m_{a|b}})_{\alpha}]
\end{align}
\begin{align}\label{fuzzy_conditional_variance}
    ({\Sigma_{a|b}})_{\alpha} = [(\underline{\Sigma_{a|b}})_{\alpha} , (\overline{\Sigma_{a|b}})_{\alpha}]. 
\end{align}
\end{proof}
\noindent
After exploring the concept of Gaussian fuzzy random variables and their utility in modeling both sources of uncertainty with Gaussian characteristics, the concept of Gaussian fuzzy stochastic process is developed, where any finite collection of fuzzy random variables is a Gaussian fuzzy random vector.
\section{Gaussian fuzzy process}\label{Gaussian fuzzy process}
Stochastic processes are mathematical models used to describe systems that evolve over time in a probabilistic manner. They are collections of random variables indexed by time or space, capturing the inherent randomness in dynamic systems. One of the most commonly studied stochastic processes is the Gaussian process, where any finite collection of random variables has a joint Gaussian distribution. The traditional stochastic processes effectively model randomness but struggle with the vagueness inherent in many real-world problems, such as linguistic uncertainty or imprecise measurements. Fuzzy stochastic processes are needed to address the situations where both randomness and imprecision coexist. By incorporating fuzzy logic into stochastic processes, fuzzy stochastic processes can better handle these dual sources of uncertainty. A general theory of fuzzy stochastic processes is developed in \cite{wang1992theory}.
\subsection{Fuzzy stochastic processes}
\begin{definition}
    A fuzzy random function $X(t)= \{X^t\}_{t \in T}$ is a fuzzy set-valued function from $T \times \Omega$ to $\mathcal{F}(\mathbb{R})$. For each $t \in T, X^t$ is a fuzzy random variable, where $T = \mathbb{Z}, \mathbb{Z}^{+}, \mathbb{R}, \mathbb{R}^{+}, [a,b], etc.$ If $T = \mathbb{R}$ or $\mathbb{R}^{+} $ or $ [a,b]$, then $X(t)$ is called a fuzzy stochastic process.
\end{definition}
Like in the classical stochastic process

\begin{figure}[h]
    \centering
        \includegraphics[width=0.5\textwidth]{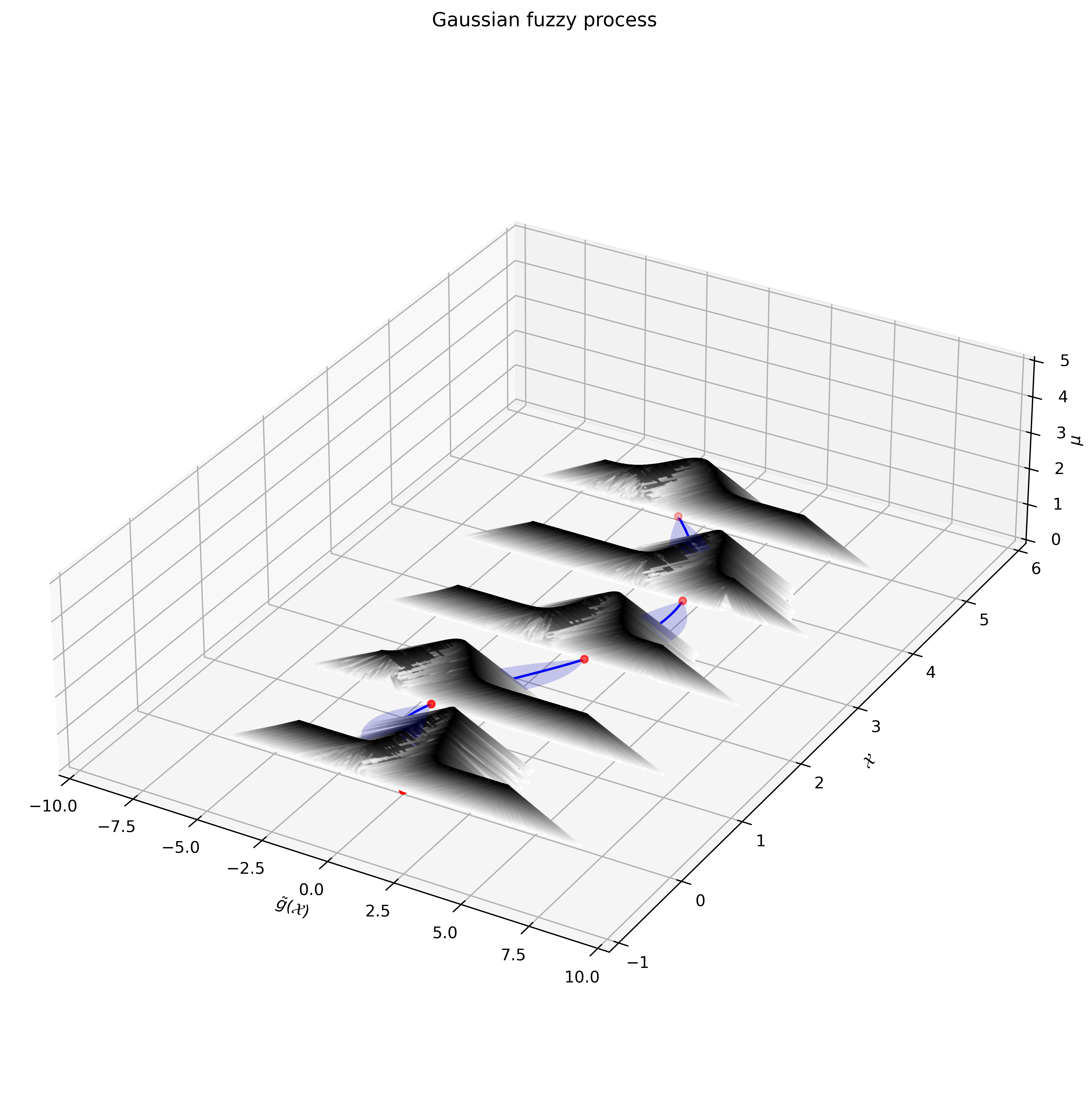}
         \caption{Example of representation of a Gaussian fuzzy process}
         \label{fig:GFP}
\end{figure}
\subsection{Gaussian fuzzy process}
\begin{definition}
     A Gaussian fuzzy process is a collection of fuzzy random variables, and any finite collection of them is a fuzzy Gaussian random vector.
\end{definition}
The Gaussian fuzzy process is completely determined by a fuzzy-valued mean function $\tilde{m}(x)$ and a fuzzy-valued kernel function $\tilde{k}(x,x^{'})$ (by taking $s$ in the kernel function as fuzzy). Assume a function $\tilde{g} : \mathcal{X} \rightarrow{} \mathcal{F}(\mathbb{R})$ follows Gaussian fuzzy process, i.e., $\tilde{g}(x) \sim FGP(\tilde{m} , \tilde{k})$. Here, $\tilde{m}(x) : \mathcal{X} \rightarrow \mathcal{F}(\mathbb{R}) $, defined as $\tilde{m}(x) = E[\tilde{g}(x)]$,  is the mean function and $\tilde{k} : \mathcal{X} \times \mathcal{X} \rightarrow \mathcal{F}(\mathbb{R})$, is covariance or kernel function. Here, $\tilde{m}(x)$ is the expected fuzzy functional value of all functions present in the distribution at the point $x$, and $\tilde{k}(x, x^{'})$ represents the dependence between the function values at different input points. 
\par Consider a finite collection of $n$ points, $x_{1},x_{2},\ldots, x_{n} \in\mathcal{X}$. Then their function values are $\tilde{g}_{i} = \tilde{g}(x_{i})$, for $i = 1,2,\ldots, n$. From the assumption $\tilde{g}_{1:n} = [\tilde{g}_{1}, \tilde{g}_{2}, \ldots , \tilde{g}_{n}]$ are Jointly Gaussian with the mean vector  $m_{\tilde{g}_{1:n}} = [\tilde{m}_{g_1}, \tilde{m}_{g_2}, \ldots , \tilde{m}_{g_n}]$ and covariance matrix $\boldsymbol{\Sigma}_{\tilde{g}_{1:n}\tilde{g}_{1:n}} = \tilde{k}(x_{i}, x_{j})$. 
\begin{theorem}\label{th:posterior_mean_variance}
    For any point $x_* \in \mathcal{X}$ its functional value $\tilde{g}(x)$ is a Gaussian fuzzy random variable given $\tilde{g}_{1:n}$ is a Gaussian fuzzy random vector with $\tilde{g}(x)$ having fuzzy mean and fuzzy variance given by their $\alpha$-cuts as
    \begin{align*}
        (m_{\tilde{g} | \tilde{g}_{1:n}}(x_*))_{\alpha} = [(\underline{m_{\tilde{g} | \tilde{g}_{1:n}}})_{\alpha} , (\overline{m_{\tilde{g} | \tilde{g}_{1:n}}})_{\alpha}]
    \end{align*}
    \begin{align*}
        (\Sigma_{\tilde{g} | \tilde{g}_{1:n}}(x_*))_{\alpha} = [(\underline{\Sigma_{\tilde{g} | \tilde{g}_{1:n}}})_{\alpha} , (\overline{\Sigma_{\tilde{g} | \tilde{g}_{1:n}}})_{\alpha}], 
    \end{align*}
where 
\begin{align*}
    (\underline{m_{\tilde{g} | \tilde{g}_{1:n}}})_{\alpha} &= (\underline{m_{\tilde{g}}})_{\alpha} + (\underline{\boldsymbol{\Sigma}_{\tilde{g}\tilde{g}_{1:n}}})_{\alpha}(\underline{\boldsymbol{\Sigma}_{\tilde{g}_{1:n}\tilde{g}_{1:n}}})_{\alpha}^{-1}((\underline{X_b})_{\alpha} - (\underline{m_{\tilde{g}_{1:n}}})_{\alpha})\\
     (\underline{\Sigma_{\tilde{g} | \tilde{g}_{1:n}}})_{\alpha} &= (\underline{\boldsymbol{\Sigma}_{\tilde{g}\tilde{g}}})_{\alpha} - (\underline{\boldsymbol{\Sigma}_{\tilde{g}\tilde{g}_{1:n}}})_{\alpha}(\underline{\boldsymbol{\Sigma}_{\tilde{g}_{1:n}\tilde{g}_{1:n}}})_{\alpha}^{-1}(\underline{\boldsymbol{\Sigma}_{\tilde{g}_{1:n}\tilde{g}}})_{\alpha}\\
     (\overline{m_{\tilde{g} | \tilde{g}_{1:n}}})_{\alpha} &= (\overline{m_{\tilde{g}}})_{\alpha} + (\overline{\boldsymbol{\Sigma}_{\tilde{g}\tilde{g}_{1:n}}})_{\alpha}(\overline{\boldsymbol{\Sigma}_{\tilde{g}_{1:n}\tilde{g}_{1:n}}})_{\alpha}^{-1}((\overline{X_b})_{\alpha} - (\overline{m_{\tilde{g}_{1:n}}})_{\alpha})\\
     (\overline{\Sigma_{\tilde{g} | \tilde{g}_{1:n}}})_{\alpha} &= (\overline{\boldsymbol{\Sigma}_{\tilde{g}\tilde{g}}})_{\alpha} - (\overline{\boldsymbol{\Sigma}_{\tilde{g}\tilde{g}_{1:n}}})_{\alpha}(\overline{\boldsymbol{\Sigma}_{\tilde{g}_{1:n}\tilde{g}_{1:n}}})_{\alpha}^{-1}(\overline{\boldsymbol{\Sigma}_{\tilde{g}_{1:n}\tilde{g}}})_{\alpha}
\end{align*}
\end{theorem}
\begin{proof}
    Follows from Theorem \ref{th:conditional_distribution}.
\end{proof}

\noindent
After discussing and developing all the necessary theoretical background about fuzzy random variables, Gaussian fuzzy random variables, vectors, and Gaussian fuzzy process, this work next proposes the novel Bayesian fuzzy optimization for optimizing expensive fuzzy-valued objective functions after incorporating different kinds of uncertainty associated with the model.
\section{Bayesian fuzzy optimization}\label{Bayesian fuzzy optimization} 
Bayesian fuzzy optimization is an extension of traditional Bayesian optimization that handles fuzzy-valued objective functions. Traditional Bayesian optimization is a powerful method used for optimizing expensive or black-box functions, often relying on a Gaussian process as a surrogate model to predict the function's behavior. In cases where the objective function is fuzzy-valued, meaning that the outputs are not exact numbers but rather fuzzy numbers, Bayesian optimization must be adapted to account for this fuzziness. This leads to the development of BFO, where the prior belief about the objective function or surrogate model of the objective function is taken to be a Gaussian fuzzy process discussed in Section \ref{Gaussian fuzzy process}. If there is a fuzzy-valued function $\tilde{g}$ that is to be optimized using BFO, the following algorithm gives the steps for the optimization method. 
\subsection{Surrogate model selection}
In Bayesian optimization, a surrogate model is used to define a prior belief about the objective function, guiding the search for the optimum by predicting the function's behavior in unexplored regions. For fuzzy-valued objective functions, generally, the Gaussian fuzzy process serves as a natural extension of the traditional Gaussian process to handle fuzziness with randomness. One can choose different surrogate models depending on the problem.
\subsection{Fuzzy acquisition function}
Like in the traditional BO, here in BFO, the fuzzy acquisition functions determine how the search space should be searched with the help of the posterior fuzzy mean and fuzzy variance. For doing that, the following acquisition functions are generalized.
\subsubsection{Fuzzy expected improvement}
Fuzzy expected improvement is an extension of the traditional expected improvement acquisition function (see Subsection \ref{EI}) used in traditional Bayesian optimization, adapted to handle fuzzy-valued objective functions. The goal of FEI is to guide the search for the optimum by considering the fuzziness in both the model predictions and the objective function. Let $\tilde{f}(x)$  denote the fuzzy value of the objective function at point $x$, and let $\tilde{y}^*$ be the current best fuzzy observation. Then, the FEI can be defined as the expected value of the improvement over $\tilde{y}^*$, considering the fuzzy nature of the predictions.
Mathematically, the FEI can be represented as:
\begin{align}
    \text{FEI}(x) = E[\max(\tilde{y}^* - \tilde{f}(x), 0)]
\end{align}
\subsubsection{Fuzzy upper confidence bound}
\emph{Fuzzy upper confidence bound} (FUCB) is an adaptation of the traditional \emph{upper confidence bound} (UCB) acquisition function discussed in Section \ref{ucb}, which is used in classical Bayesian optimization, tailored to handle fuzzy-valued objective functions. The FUCB balances exploration and exploitation in a simple but effective way.
Let $\tilde{f}(x)$ denote the fuzzy prediction of the objective function at point $x$, where $\tilde{m}(x)$ and $\tilde{\sigma}(x)$ represent the fuzzy mean and fuzzy standard deviation of the prediction, respectively. The FUCB at a point $x$ is defined by 
\begin{align}
    \text{FUCB}(x) = \tilde{m}(x) + \beta \tilde{\sigma}(x). 
\end{align}
Here, $\beta$ is a positive parameter that controls the balance between exploration and exploitation, and the addition is performed using fuzzy arithmetic. The FUCB encourages the selection of points that have either a high fuzzy predicted mean (exploitation) or a high fuzzy predicted uncertainty (exploration), thereby guiding the search for the optimum in the presence of fuzziness in both the model predictions and the objective function.

\begin{algorithm}
\caption{Bayesian fuzzy optimization}
\label{alg:fbo}
\begin{algorithmic}[1]
\State{\textbf{Initialization:} Start with a set of initial $n$ points, say $x_{1},x_{2},\ldots, x_{n} \in\mathcal{X}$ and evaluate the fuzzy objective function values at these points, which are $\tilde{g}_i= \tilde{g}(x_i)$.}
\State{\textbf{Model building:} Assume fuzzy valued objective function $\tilde{g}$ follows a Gaussian fuzzy process as prior belief. So, construct a fuzzy Gaussian process using the initial data. This model predicts the fuzzy output for any input as discussed in Theorem \ref{th:posterior_mean_variance}.}
\State{\textbf{Optimizing acquisition function:} Use a fuzzy acquisition function (e.g., FEI, FUCB) to select the next point to evaluate.}
\State{\textbf{Evaluation:} Evaluate the fuzzy objective function at the new point and update the surrogate model with this new information to get the new posterior fuzzy mean and variance.}
\State{\textbf{Repeat:} Repeat the process, iteratively selecting new points to evaluate based on the updated posterior fuzzy mean, variance, and acquisition function, until a stopping criterion is met.}
\State{\textbf{Return:} Point giving best objective function value}
\end{algorithmic}
\end{algorithm}
BFO extends the traditional BO by incorporating fuzzy sets, allowing it to handle fuzzy-valued objective functions effectively. This integration enables the optimization of problems where outputs are not crisp numbers but fuzzy values, reflecting real-world scenarios with uncertainty associated with both fuzziness and randomness. The introduction of GFP as a prior belief in BFO provides a sophisticated tool for approximating fuzzy objective functions. The GFP is actually based on fuzzy random variables, which are naturally effective in capturing uncertainty, offering a more accurate representation of the objective function’s behavior. The adaptation of acquisition functions, such as FEI and FUCB, enables effective exploration and exploitation in the presence of fuzziness. These functions guide the search process by balancing the trade-off between exploring uncertain regions and exploiting known promising areas. BFO is valuable in various domains, including finance, engineering, etc., where decision-making involves fuzzy or uncertain parameters. The methodology’s ability to handle fuzzy data makes it suitable for optimizing portfolios and designing complex systems, among other applications.
\section{Applications of Bayesian fuzzy optimization}\label{Application of Bayesian fuzzy optimization}

The proposed Bayesian fuzzy optimization technique is applied to two different kinds of fuzzy-valued optimization problems. One is the fuzzy mean-variance portfolio allocation problem, where the task is to find optimal weights to be allocated to a collection of assets that maximize the portfolio return and minimize the portfolio risk. Another is the optimization and analysis of fuzzy-valued temperature data to find the maximum temperature and its occurrence year. Future predictions for the temperature can also be made using Bayesian fuzzy optimization. In the portfolio allocation problem, the objective function has a mathematical expression concerning risk and return, as discussed in Section \ref{formulation}, but in the climate problem, the fuzzy-valued objective function does not have any mathematical expression. In both problems, the proposed method performed well in spite of having very little knowledge about them, which will be discussed later in this section.

\subsection{Fuzzy mean-variance portfolio allocation problem formulation}\label{formulation}

Consider a portfolio of $n$ assets; the goal is to find optimal asset weights. Assume that $R^{1}, R^{2}, \ldots R^{n}$ are fuzzy random variables denoting the return for the individual assets. Taking the fuzzy expectations of the fuzzy random variables, the individual expected return of the assets is found to be, $\pi = [\tilde{\pi_{1}}, \tilde{\pi_{2}},\ldots, \tilde{\pi_{n}}] $, where $\tilde{\pi_{1}} = E[R^{i}]$ for all $ i= 1,2 \ldots, n$. If $\omega= [\omega_{1}, \omega_{2}, \ldots, \omega_{n}]$ is the weight vector associated with the assets, the fuzzy random variable denoting the total portfolio return will be $R = \omega_{1}R^{1}+ \omega_{2}R^{2}+ \ldots +\omega_{n}R^{n}$.

A single-period fuzzy mean-variance portfolio allocation problem is defined as follows
\begin{align*}
& \mbox{maximize}~~ \tilde{\mathcal{R}} = \sum_{i=1}^{n} \omega_{i}\tilde{\pi_{i}}\\ 
& \mbox{minimize}~~ \tilde{\sigma} = \sum_{i=1}^{n}\sum_{j=1}^{n} \omega_{i}\omega_{j}\tilde{\sigma}_{ij}\\ 
&\mbox{subject to}~~ \sum_{i=1}^{n} \omega_{i} =1 , \hspace{3mm}\omega_{i} \geq 0 \hspace{3mm}\forall{i}= 1,2, \ldots, n.
\end{align*}
 
\noindent
The weighted sum technique is used to formally define the objective function, maximizing the total expected portfolio return and minimizing the total expected risk, with weights $\lambda$ and $(1- \lambda)$, respectively. This factor indicates the importance of that particular objective.

\begin{align}\label{eq:objective_function}
& \mbox{maximize}~~ \tilde{g}(\omega)= \lambda \tilde{\mathcal{R}}-(1-\lambda)\tilde{\sigma}= \lambda \sum_{i=1}^{n} \omega_{i}\tilde{\pi_{i}} - (1 - \lambda) \sum_{i=1}^{n}\sum_{j=1}^{n} \omega_{i}\omega_{j}\tilde{\sigma}_{ij}  \\ 
& \mbox{subject to}~~\sum_{i=1}^{n} \omega_{i} =1 , \hspace{3mm}\omega_{i} \geq 0 \hspace{3mm}\forall{i}= 1,2, \ldots, n, \notag 
\end{align}
where $\lambda$ is the scalarization constant. Here, the formulated objective function is a fuzzy-valued function that has to be optimized. The optimization method discussed in Section \ref{Bayesian fuzzy optimization} will be used to solve the problem. 

\subsection{Data and performance measures}
For the experiment purpose of the portfolio allocation problem, the lowest, closing, and highest prices for 10 stocks from different sectors of the Indian National Stock Exchange for each day of five years (from 1st January 2019 to 1st January 2024) are collected. Stocks are chosen from different market sectors like communications, automobiles, retail, energy,  information technology, paints, banking, power, oil and gas, etc, to gain knowledge from the diverse market fields. Returns for each stock for each day are considered as triangular fuzzy numbers (lowest, closing, highest). The returns are assumed to be fuzzy random variables, which not only consider the existing randomness in the problems but also the imprecision of the return prices of the stocks. Next, the fuzzy valued mean-variance portfolio allocation problem is formulated as discussed in Section \ref{formulation}, Equation (\ref{eq:objective_function}). 
\par In the climate application, the historical data about the annual lowest, mean, and highest temperatures in India is collected from the Open Government Data Platform India website. The temperature for the years is considered as triangular fuzzy numbers (lowest, mean, highest). This fuzzy-valued annual temperature function (treated as an objective function) has no mathematical expression. Only a few data points are fed into the optimization method, not the whole dataset. The task is to find the year when the maximum temperature (treated as a fuzzy number) will occur and the temperature itself. Also, the proposed Fuzzy Gaussian process is tested as a regressor using this dataset, and a comparison of its performance with benchmark methods is also provided.

\begin{table*}[htbp]
\centering
\caption{Summary of datasets, experiments, benchmark methods, and key hyperparameters.}
\label{tab:experiment_summary_1}
\small
\setlength{\tabcolsep}{4pt}
\renewcommand{\arraystretch}{1.2}
\resizebox{\textwidth}{!}{
\begin{tabular}{p{4.3cm}| p{4.0cm}| p{3.7cm}| p{4.4cm}}
\hline
\textbf{Dataset description} 
& \textbf{Experiments} 
& \textbf{Methods/Models} 
& \textbf{Key hyperparameters} \\
\hline
\textbf{1:} 5-year data (2019--2024) of 10 stocks from the National Stock Exchange India; (low, close, high) as triangular fuzzy numbers from Yahoo Finance. Tickers of the used stocks are TCS.NS, INFY.NS, RELIANCE.NS, HDFCBANK.NS, LT.NS, HINDUNILVR.NS, ASIANPAINT.NS, COALINDIA.NS, NTPC.NS, TATAMOTORS.NS.
\par
\textbf{2:} Indian annual temperature data (1901--2021); (low, mean, high) as triangular fuzzy numbers from the Open Government Data  platform India.

& \textbf{1:} Comparison of performance of proposed Bayesian Fuzzy Optimization with benchmark methods for portfolio optimization application.
\par
\textbf{2:} Comparison of Fuzzy Gaussian Process with benchmarks using weather data.
\par
\textbf{3:} Observing the performance of BFO for weather data, treating the objective as completely unknown and finding the highest temperature with very few known points.
\par
\textbf{4:} Statistical significance test with a different set of hyperparameters.
& For financial application: Proposed Bayesian Fuzzy Optimization, Bayesian Optimization, Differential Evolution, Genetic Algorithm, Particle Swarm Optimization, Ant Colony Optimization, Artificial Bee Colony.
For climate application: Gaussian Process, Random Forest, Support Vector Regression, Gradient Boosting Regression, and Multi-Layer Perceptron Regression.
& For all experiments and all methods (where applicable), the key hyperparameters are used as: number of initial samples = 10, length scale in kernel function = 0.5, scalarization constant = 0.5, number of trees = 100, maximum depth = 5, learning rate = 0.1, hidden units = 32, number of estimators =100. All the other hyperparameters are taken as default, as provided in the Python library sklearn. Fuzzy Expected Improvement is used as the default acquisition function. Fuzzy Upper Confidence Bound is used in a statistical significance test for the climate data.
\\
\hline
\end{tabular}
}
\end{table*}

\subsection{Results and discussion}
This section presents the experimental results of the proposed Bayesian fuzzy optimization method on two applications: (i) fuzzy mean--variance portfolio allocation and (ii) fuzzy-valued temperature modeling and optimization. These experiments demonstrate the effectiveness of the proposed method in handling uncertainty arising from both randomness and fuzziness, while maintaining strong performance under limited observations. All the details regarding the experiments and hyperparameter settings are provided in Table \ref{tab:experiment_summary_1}.

\subsection{Portfolio optimization results}

Table (\ref{tab:comparison}) presents the comparative performance of BFO against several benchmark optimization algorithms under a fixed evaluation budget of 50 function evaluations over 30 independent runs. It can be observed that the proposed BFO achieves the highest mean objective value (4.687) among all competing methods, along with a relatively low standard deviation (0.225), indicating both superior performance and consistency.

The statistical significance analysis further reveals that BFO performs significantly better than DE, GA, PSO, ABC, and ACO, while showing comparable performance with classical Bayesian optimization (BO). However, the higher mean AUC (area under the curve) value obtained by BFO demonstrates faster convergence, highlighting its efficiency in both exploration and exploitation of the search space.
\begin{table}[h]
\centering
\caption{Performance comparison of optimization methods under fixed evaluation budget(50) for 30 independent runs}
\label{tab:comparison}
\begin{tabular}{lcccc}
\hline
Method & Mean Objective & Std Dev & Statistical Significance & Mean AUC \\
\hline
BFO & 4.687 & 0.225 & --         & 203.11 \\
BO  & 4.472 & 0.286 & $\approx$  & 186.32  \\
DE  & 4.006 & 0.290 & +          & 104.65  \\
GA  & 4.231 & 0.345 & +          & 117.22  \\
PSO & 4.396 & 0.360 & +          & 136.87  \\
ABC & 4.304 & 0.255 & +          & 128.63  \\
ACO & 4.110 & 0.165 & +          & 110.30  \\
\hline
\end{tabular}
\end{table}

\begin{table}
   
\begin{center}
\centering
\scriptsize
\begin{tabular}{l|l|l|l|l|l|l|l}

\hline
$\lambda$  & $n_0$ & $l$ & Mean objective value & Return & Risk &  Sharpe ratio & $E_{total}$ \\
\hline
 0.1 & 5 & 1 & 4.3814 & 10.8862 & 2.1234 & 5.1267 & 0.2875 \\
0.2 & 5 & 1 & 4.4383 & 10.7800 & 1.9034 & 5.6634 & 0.2510 \\
0.3 & 5 & 1 & 4.5842 & 10.8320 & 1.6637 & \textbf{6.5106} & 0.3685 \\
0.5 & 5 & 1 & 4.4550 & 10.9120 & 2.0021 & 5.4500 & 0.2484 \\
0.6 & 5 & 1 & 4.5604 & 10.6651 & 1.5443 & \textbf{6.9059} & 0.5292 \\
0.7 & 5 & 1 & 4.5181 & 10.9796 & 1.9434 & 5.6497 & 0.4621 \\
0.8 & 5 & 1 & 4.7022 & 11.0066 & 1.6022 & \textbf{6.8695} & 0.6014 \\
0.9 & 5 & 1 & 4.6653 & 10.8239 & 1.4934 & \textbf{7.2475} & 0.4986 \\

0.2 & 4 & 1 & 4.3832 & 10.8108 & 2.0444 & 5.2879 & 0.4031 \\
0.3 & 4 & 1 & 4.5151 & 11.1074 & 2.0773 & 5.3469 & 0.4691 \\
0.4 & 4 & 1 & 4.5983 & 11.0014 & 1.8049 & 6.0950 & 0.3169 \\
0.5 & 4 & 1 & 4.4294 & 10.9313 & 2.0725 & 5.2743 & 0.3128 \\
0.6 & 4 & 1 & 4.8564 & 11.1416 & 1.4288 & \textbf{7.7975} & 0.4894 \\
0.7 & 4 & 1 & 4.5348 & 10.7278 & 1.6583 & 6.4689 & 0.5512 \\
0.8 & 4 & 1 & 4.6063 & 11.0969 & 1.8843 & 5.8890 & 0.5356 \\
0.9 & 4 & 1 & 4.4710 & 10.6511 & 1.7091 & 6.2318 & 0.5560 \\

0.1 & 6 & 1 & 4.2789 & 10.7116 & 2.1539 & 4.9730 & 0.5346 \\
0.2 & 6 & 1 & 4.4408 & 10.6786 & 1.7970 & 5.9424 & 0.4597 \\
0.3 & 6 & 1 & 4.3819 & 10.7248 & 1.9611 & 5.4685 & 0.2552 \\
0.4 & 6 & 1 & 4.4264 & 10.8212 & 1.9685 & 5.4972 & 0.1903 \\
0.5 & 6 & 1 & 4.6104 & 10.7438 & 1.5231 & \textbf{7.0537} & 0.6340 \\
0.6 & 6 & 1 & 4.6102 & 11.0817 & 1.8614 & 5.9534 & 0.4698 \\
0.7 & 6 & 1 & 4.4889 & 10.9303 & 1.9525 & 5.5980 & 0.4735 \\
0.9 & 6 & 1 & 4.5603 & 11.2799 & 2.1593 & 5.2237 & 0.4662 \\

0.1 & 3 & 1 & 4.5901 & 10.8458 & 1.6657 & \textbf{6.5110} & 0.2792 \\
0.2 & 3 & 1 & 4.5164 & 10.7111 & 1.6784 & 6.3813 & 1.0052 \\
0.3 & 3 & 1 & 4.3591 & 10.6001 & 1.8820 & 5.6322 & 1.1817 \\
0.4 & 3 & 1 & 4.2705 & 10.8034 & 2.2624 & 4.7750 & 1.4596 \\
0.5 & 3 & 1 & 4.2802 & 10.7270 & 2.1666 & 4.9510 & 0.9133 \\
0.6 & 3 & 1 & 4.6677 & 10.7973 & 1.4619 & \textbf{7.3853} & 0.7994 \\
0.7 & 3 & 1 & 4.4107 & 10.8944 & 2.0731 & 5.2549 & 0.7168 \\
0.9 & 3 & 1 & 4.2697 & 10.7986 & 2.2592 & 4.7797 & 1.2165 \\

0.1 & 5 & 1.5 & 4.4078 & 11.0426 & 2.2270 & 4.9584 & 0.4940 \\
0.2 & 5 & 1.5 & 4.4804 & 11.1487 & 2.1879 & 5.0954 & 0.9179 \\
0.3 & 6 & 1.5 & 4.4765 & 10.9265 & 1.9736 & 5.5360 & 0.7350 \\
0.4 & 4 & 1.5 & 4.5213 & 10.9068 & 1.8643 & 5.8501 & 0.6820 \\
0.5 & 3 & 1.5 & 4.4814 & 10.7782 & 1.8155 & 5.9367 & 0.7641 \\
0.6 & 6 & 1.5 & 4.4142 & 10.8790 & 2.0506 & 5.3051 & 0.8041 \\
0.7 & 5 & 1.5 & 4.7094 & 11.1666 & 1.7479 & 6.3885 & 0.4343 \\
0.9 & 3 & 1.5 & 4.5167 & 10.8666 & 1.8333 & 5.9273 & 0.8385 \\

0.1 & 5 & 0.5 & 4.4711 & 11.0455 & 2.1034 & 5.2512 & 0.4997 \\
0.2 & 6 & 0.5 & 4.3874 & 10.5801 & 1.8054 & 5.8600 & 0.8815 \\
0.3 & 5 & 0.5 & 4.4525 & 10.9303 & 2.0254 & 5.3965 & 0.9853 \\
0.4 & 4 & 0.5 & 4.5588 & 11.0274 & 1.9099 & 5.7736 & 0.5955 \\
0.5 & 3 & 0.5 & 4.6351 & 10.9098 & 1.6397 & \textbf{6.6535} & 0.5750 \\
0.6 & 6 & 0.5 & 4.7848 & 10.8874 & 1.3178 & \textbf{8.2618} & 0.8853 \\
0.7 & 5 & 0.5 & 4.2007 & 10.6391 & 2.2378 & 4.7540 & 1.0085 \\
0.8 & 4 & 0.5 & 4.6216 & 11.1585 & 1.9153 & 5.8259 & 0.9651 \\

\hline

\end{tabular}
\caption{Performance of Bayesian Fuzzy Optimization on fuzzy mean-variance portfolio allocation problem over Indian National Stock Exchange data for different values of hyper-parameters}
\label{tab:BFO_MVPO}
\end{center}
\end{table}

A detailed sensitivity analysis is presented in Table (\ref{tab:BFO_MVPO}), where the performance of BFO is evaluated across different hyperparameter settings, including the scalarization constant ($\lambda$), number of initial observations ($n_0$), and kernel length scale ($l$).

From the table, it can be observed that the proposed method consistently achieves high mean objective values across a wide range of hyperparameter combinations, demonstrating strong robustness. For financial enthusiasts, the Sharpe ratio is used as an additional performance measure together with risk and returns, and it can be seen that they attain significantly high values, indicating an effective balance between return and risk.

Additionally, the total error $E_{total}$ remains relatively low for most configurations, confirming the accuracy of the predicted fuzzy values throughout the optimization process. Even with varying hyperparameters, the performance degradation is minimal, suggesting that the method is not highly sensitive to hyperparameter tuning. $E_{total}$ represents the total average distances (as defined in Equation \ref{eq:distance}) between the predicted fuzzy values and actual fuzzy values fed to the models in the iterations of the training process of the optimization method.

Figure (\ref{fig:risk_return_curve}) shows the relationship between fuzzy returns and fuzzy risks. The curve highlights the trade-off between return and risk, and it can be observed that the solutions obtained by the proposed method lie in the efficient region, demonstrating its capability to identify optimal portfolios under uncertainty.
\begin{figure}[h]
    \centering
        \includegraphics[width=0.5\textwidth]{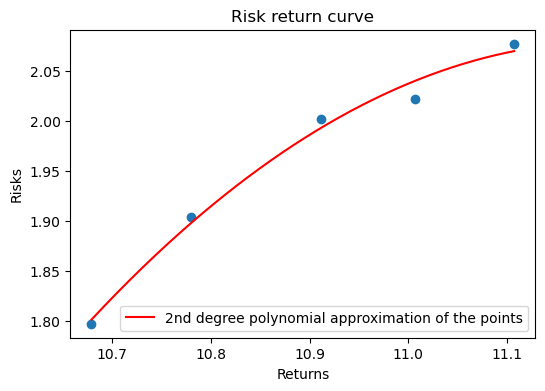}
         \caption{Return vs Risk measure curves for fuzzy mean-variance portfolio allocation problem on Indian National Stock Exchange data}
         \label{fig:risk_return_curve}
\end{figure}
\subsection{Climate data application}

In the second application, the objective function representing temperature is treated as a completely unknown fuzzy-valued function. Only a limited number of observations are provided to the model, and the task is to identify the optimal value. And also test the performance of the proposed Fuzzy Gaussian Process technique.

\subsubsection{Regression performance}

Table (\ref{tab:regression_results}) presents the comparison of the proposed fuzzy Gaussian process (FGP) with several standard regression models.
\begin{table}[h]
\centering
\caption{Regression performance comparison on temperature dataset}
\label{tab:regression_results}
\begin{tabular}{lcc}
\toprule
Method & RMSE & MAE \\
\midrule
FGP & 0.2421 & 0.2375 \\
GP  & 0.2574 & 0.2566 \\
SVR & 0.5817  & 0.5451  \\
RF  & 0.3775  & 0.3280  \\
GBR & 0.3889  & 0.3402  \\
MLP & 0.5669  & 0.5516  \\

\bottomrule
\end{tabular}
\end{table}
It can be observed that FGP achieves the lowest RMSE (0.2421) and MAE (0.2375) among all methods. It also outperforms the classical Gaussian Process (GP), demonstrating the advantage of incorporating fuzziness in addition to stochastic modeling.

Tree-based models such as Random Forest (RF) and Gradient Boosting Regression (GBR) perform reasonably well but are inferior to FGP, as they do not explicitly model uncertainty. Support Vector Regression (SVR) and Multi-Layer Perceptron (MLP) show comparatively weaker performance, likely due to their sensitivity to hyperparameters and limited capability in capturing smooth temporal trends.

The closeness of RMSE and MAE values for FGP indicates that the prediction errors are uniformly distributed without significant outliers, confirming the stability of the model.

\subsubsection{Optimal value prediction}

The sensitivity analysis results of predicting the optimal temperature values are shown in Table (\ref{tab:BFO_Temp}). It can be observed that in most cases, the predicted optimal value matches the actual optimal value (mean temperature $\approx 26.2$).
\begin{table}
\begin{center}
\scriptsize
\begin{tabular}{l|l|l|l|l}
\hline
$\beta$ & $n_0$ & $l$ & Predicted optimal & $E_{total}$ \\
\hline

0.3 & 5 & 1 & 26.05 & 1.04798 \\
0.5 & 5 & 1 & 26.20 & 1.15948 \\
0.7 & 5 & 1 & 26.20 & 0.72242 \\
0.8 & 5 & 1 & 26.04 & 0.41247 \\
0.9 & 5 & 1 & 26.20 & 1.27134 \\

0.4 & 4 & 1 & 26.20 & 1.84346 \\
0.3 & 4 & 1 & 26.20 & 0.47407 \\
0.5 & 4 & 1 & 26.05 & 0.75905 \\
0.6 & 4 & 1 & 26.05 & 0.75153 \\
0.7 & 4 & 1 & 26.20 & 0.61423 \\

0.7 & 3 & 1 & 25.62 & 0.61245 \\
0.5 & 3 & 1 & 26.20 & 0.52443 \\
0.4 & 3 & 1 & 26.20 & 0.65698 \\
0.3 & 3 & 1 & 26.04 & 0.74875 \\
0.2 & 3 & 1 & 26.20 & 0.47852 \\

0.2 & 2 & 1 & 26.05 & 0.57672 \\
0.3 & 2 & 1 & 26.05 & 0.46943 \\
0.4 & 2 & 1 & 26.20 & 0.79756 \\
0.5 & 2 & 1 & 26.20 & 1.17616 \\
0.6 & 2 & 1 & 26.20 & 0.73133 \\
0.8 & 2 & 1 & 26.20 & 0.74655 \\

0.8 & 6 & 1 & 26.20 & 0.73022 \\
0.7 & 6 & 1 & 26.20 & 0.90347 \\
0.5 & 6 & 1 & 26.20 & 0.74012 \\
0.4 & 6 & 1 & 26.20 & 1.51188 \\

0.4 & 5 & 0.5 & 26.20 & 0.90865 \\
0.6 & 5 & 0.5 & 26.05 & 1.38151 \\
0.8 & 5 & 0.5 & 26.20 & 1.05826 \\
0.2 & 5 & 0.5 & 26.20 & 1.42746 \\

0.2 & 4 & 0.5 & 26.05 & 1.00404 \\
0.5 & 4 & 0.5 & 25.65 & 0.52515 \\
0.7 & 4 & 0.5 & 26.20 & 0.67960 \\

0.8 & 3 & 0.5 & 26.05 & 0.66047 \\
0.6 & 3 & 0.5 & 26.20 & 0.88473 \\
0.5 & 3 & 0.5 & 26.20 & 1.21987 \\

0.4 & 2 & 0.5 & 26.20 & 0.92328 \\
0.2 & 2 & 0.5 & 26.05 & 0.49910 \\
0.5 & 3 & 0.5 & 26.05 & 0.60225 \\

0.7 & 2 & 1.5 & 25.45 & 0.19445 \\
0.5 & 2 & 1.5 & 25.61 & 0.39784 \\

0.4 & 3 & 1.5 & 26.20 & 0.74195 \\
0.6 & 3 & 1.5 & 26.04 & 0.75793 \\
0.8 & 3 & 1.5 & 26.20 & 0.67437 \\

0.8 & 4 & 1.5 & 26.05 & 0.77221 \\
0.6 & 4 & 1.5 & 25.93 & 0.41970 \\
0.3 & 4 & 1.5 & 26.20 & 0.99384 \\

0.3 & 5 & 1.5 & 26.20 & 1.09907 \\
0.5 & 5 & 1.5 & 26.04 & 0.85415 \\
0.7 & 5 & 1.5 & 26.20 & 0.79742 \\

\hline

\end{tabular}

\caption{Performance of Bayesian fuzzy optimization on fuzzy-valued Indian temperature data from 1901 to 2021 for different values of hyper-parameters}
\label{tab:BFO_Temp}
\end{center}
\end{table}
Even when slight deviations occur, the associated error remains small, as indicated by low values of $E_{total}$. The method is able to identify the optimal solution with very few initial observations, demonstrating strong data efficiency.

\subsubsection{Visualization and convergence}
\begin{figure}[h]
    \centering
        \subfigure[]{\includegraphics[width=0.49\textwidth]{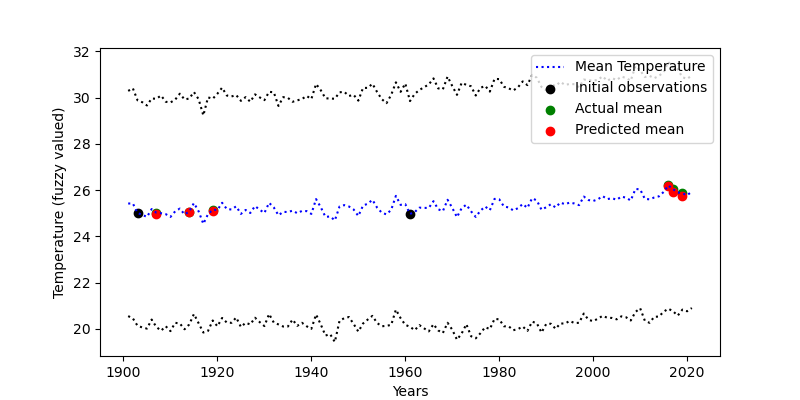}}
        \subfigure[]{\includegraphics[width=0.49\textwidth]{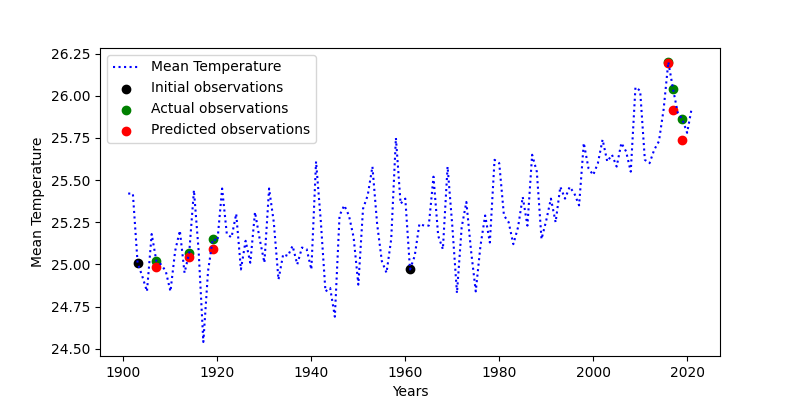}}
         \caption{Performance of the Bayesian fuzzy optimization on Indian temperature data}
         \label{fig:performance_weather}
\end{figure}
Figure (\ref{fig:performance_weather}) shows the predictive performance of the method, where the unknown temperature function is approximated using limited observations and still reaches its maximum value. Figure (\ref{fig:objective_function_prediction}) further shows that the method consistently converges to the optimal solution regardless of the number of initial observations. This indicates a well-balanced exploration-exploitation mechanism and reliable convergence behavior.
\begin{figure}
    \centering
    \subfigure[]{\includegraphics[width=0.49\textwidth]{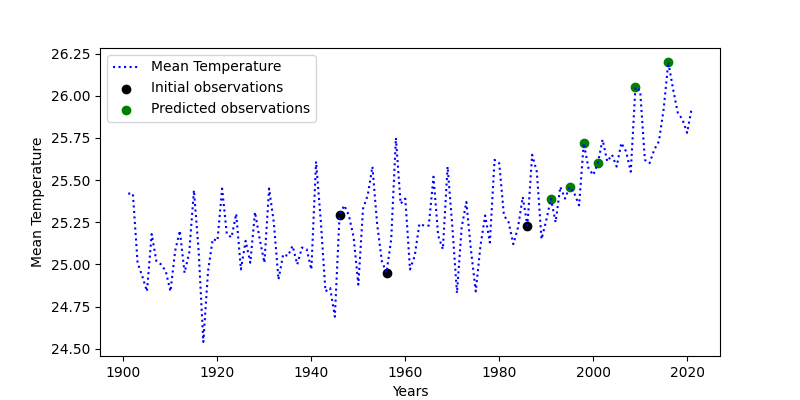}} 
    \subfigure[]{\includegraphics[width=0.49\textwidth]{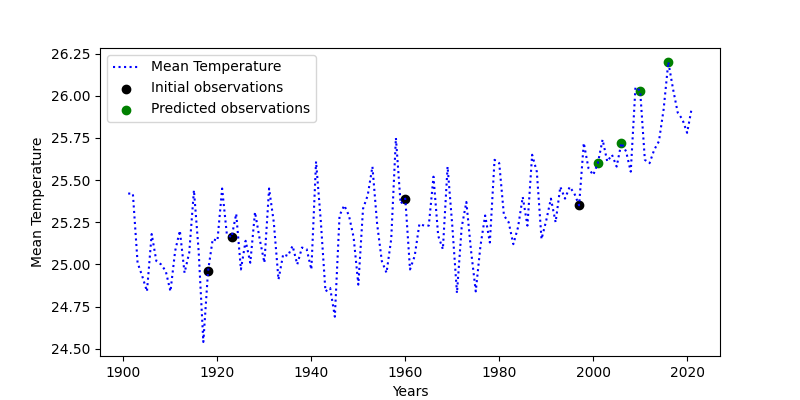}} 
    \subfigure[]{\includegraphics[width=0.49\textwidth]{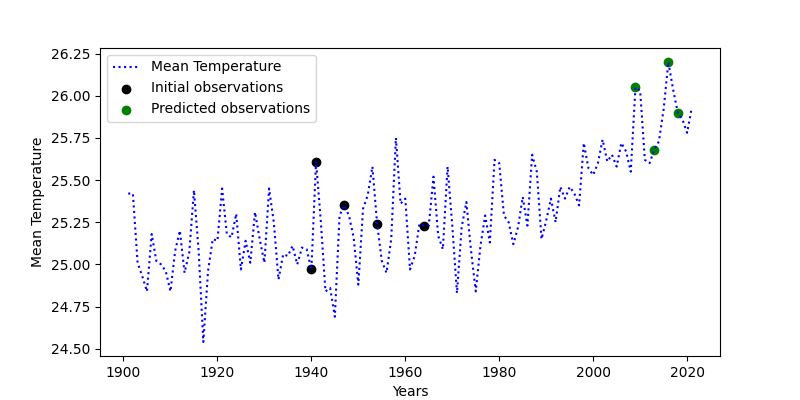}} 
    \subfigure[]{\includegraphics[width=0.49\textwidth]{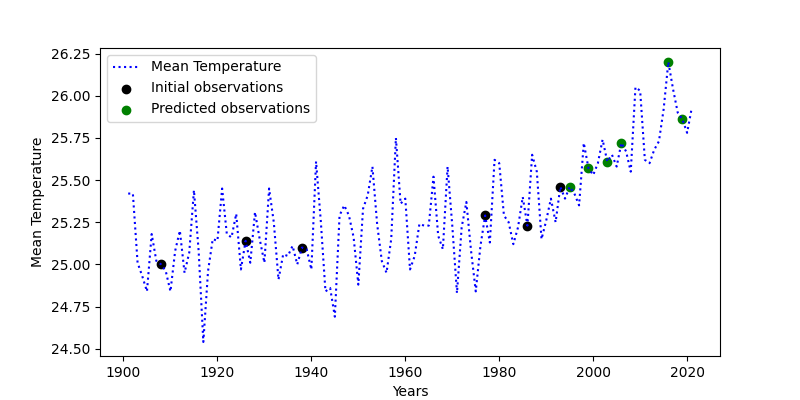}}

    \caption{Performance of Bayesian fuzzy optimization on the Indian annual temperature data (a) for three initial observations, (b) for four initial observations, (c) for five initial observations, (d) for six initial observations}
    \label{fig:objective_function_prediction}
\end{figure}
\subsubsection{Overall discussion}

The experimental results across both applications demonstrate that the proposed Bayesian fuzzy optimization method:

\begin{itemize}
    \item Achieves superior optimization performance compared to existing methods,
    \item Maintains robustness across different hyperparameter settings,
    \item Provides accurate predictions with limited data,
    \item Effectively handles uncertainty arising from both randomness and fuzziness.
\end{itemize}

The consistent performance across financial and climate datasets highlights the versatility and generalizability of the proposed Fuzzy Gaussian Process (FGP) as a regressor and Bayesian Fuzzy Optimization(BFO) as a fuzzy optimization technique. This suggests that the proposed approach can be extended to a wide range of real-world applications involving uncertain, imprecise, and expensive objective functions, such as healthcare optimization, environmental modeling, supply chain systems, and risk management.

\section{Conclusion}\label{conclusion}
This work aimed to explore the integration of both sources of uncertainty, named randomness and imprecision, through fuzzy random variables, with a focus on developing a theoretical framework for Gaussian stochastic processes. Then, proposing a novel fuzzy optimization technique, Bayesian fuzzy optimization, that can optimize a fuzzy-valued objective function. Also, applications of Bayesian fuzzy optimization to different real-world problems. The work developed a theoretical background for Gaussian stochastic processes under a fuzzy environment. An expression for the posterior fuzzy mean and fuzzy variance is derived. It is then operationalized through a novel Bayesian fuzzy optimization technique, where the fuzzy-valued objective function is assumed to follow the Gaussian fuzzy process as prior belief. Then, BFO uses the expressions derived for fuzzy expected improvement and fuzzy upper confidence bound acquisition functions for guiding the search process depending on the posterior fuzzy mean and fuzzy variance. The proposed method is then successfully applied to the fuzzy mean-variance portfolio allocation problem and the analysis and optimization of the fuzzy-valued objective function representing Indian annual temperature, demonstrating strong predictive power, robustness, and versatility. This work also establishes Fuzzy Gaussian Process as a regressor counting on both sources of uncertainty effectively. An empirical demonstration of convergence is provided for the proposed Bayesian fuzzy optimization technique. The findings of this work contribute to a deeper understanding of how fuzzy random variables can be used to model complex systems, offering new tools for decision-making in environments characterized by randomness and imprecision. The success of the method in different domains highlights its potential for broad application in fields such as finance, climate science, healthcare, material science, agriculture, and beyond.
\par This work did not test the performance of the proposed optimization method on high-dimensional application problems, which is actually a limitation for this work as well as for the classical Bayesian optimization method. As the dimension increases, the method can become computationally intensive, incorporating fuzziness into another layer of complexity. The problem of interpretability, i.e, the probabilistic nature of Bayesian approaches, combined with the imprecision of fuzzy logic, can make it challenging to understand how the model reaches its conclusions. Dimension reduction techniques or a trust region-based approach can be adapted to the optimization method to improve the performance of the proposed method for more serious dimensional problems. Despite its limitations, the proposed Bayesian fuzzy optimization technique can be fruitful in various real-world domains where uncertainty arises from both randomness and fuzziness, and the fuzzy-valued objective functions are computationally intensive or lack mathematical expressions.

\bibliographystyle{plain}
\bibliography{reference}
\end{document}